\documentclass{amsart}

\usepackage{geometry, amssymb, verbatim, esint}
\usepackage{color}

\usepackage[colorlinks=true,urlcolor=blue, citecolor=red,linkcolor=blue,
linktocpage,pdfpagelabels, bookmarksnumbered,bookmarksopen]{hyperref}
\usepackage[hyperpageref]{backref}
\usepackage{tikz}
\usepackage{soul}
\usepackage{mathrsfs}
\newtheorem{teo}{Theorem}[section]
\newtheorem{lm}[teo]{Lemma}
\newtheorem{coro}[teo]{Corollary}
\newtheorem{prop}[teo]{Proposition}
\newtheorem*{main}{Main Theorem}
\theoremstyle{definition}
\newtheorem{defi}[teo]{Definition}

\newtheorem{rem}[teo]{Remark}
\newtheorem*{ack}{Acknowledgments}

\numberwithin{equation}{section}

\title[low eigenvalues]{Low eigenvalues of the $p-$Laplacian\\ in general open sets}
\date{\today}
\author[Brasco]{Lorenzo Brasco}
\address[L.\ Brasco]{Dipartimento di Matematica e Informatica
	\newline\indent
	Universit\`a degli Studi di Ferrara
	\newline\indent
	Via Machiavelli 35, 44121 Ferrara, Italy}
\email{lorenzo.brasco@unife.it}

\author[Briani]{Luca Briani}
\address[L.\ Briani]{School of Computation, Information and Technology,		\newline\indent
 	Technical University of Munich
 	\newline\indent
	Boltzmannstra\ss e 3, 85748 Garching bei M\"unchen, Germany.}
\email{luca.briani@tum.de}

\author[Prinari]{Francesca Prinari}

\address[F. Prinari]{Dipartimento di Scienze Agrarie, Alimentari e Agro-ambientali
\newline\indent 
Universit\`a di Pisa
\newline\indent
Via del Borghetto 80, 56124 Pisa, Italy}
\email{francesca.prinari@unipi.it}


\begin{document}

\begin{abstract}
We consider the minmax Ljusternik-Schnirelmann levels of the constrained $p-$Dirichlet integral, on a general open set of the Euclidean space. We show that, whenever one of these levels lies below the threshold given by the $L^p$ Poincar\'e constant ``at infinity'', it actually defines an eigenvalue of the Dirichlet $p-$Laplacian. We also prove an exponential decay at infinity for the relevant eigenfunctions: this can be seen as a \v{S}nol-Simon--type estimate for the nonlinear case. Finally, we exhibit some peculiar examples of unbounded open sets to which our main result applies.
\end{abstract}

\maketitle

\begin{center}
\begin{minipage}{10cm}
\small
\tableofcontents
\end{minipage}
\end{center}

\section{Introduction}

\subsection{Beneath the essential spectrum}
Let us start by considering the spectrum of the Dirichlet-Laplacian on an open set $\Omega\subseteq\mathbb{R}^N$. Whenever the embedding
\begin{equation}
\label{mimmergo}
W^{1,2}_0(\Omega)\hookrightarrow L^2(\Omega),
\end{equation}
is compact, we know that this spectrum is discrete, made of a diverging sequence of positive eigenvalues $\{\lambda_k(\Omega)\}_{k\in\mathbb{N}\setminus\{0\}}$, each one having a finite multiplicity. Characterizations of open sets for which the embedding \eqref{mimmergo} is compact can be found for example in \cite[Chapter 15]{Maz}.
\par
Moreover, in this situation we can build a system of orthonormal (in $L^2$ sense) associated eigenfunctions $\{u_k\}_{k\in\mathbb{N}\setminus\{0\}}\subseteq W^{1,2}_0(\Omega)$, which gives a Hilbertian basis of $L^2(\Omega)$. Accordingly, we can characterize the eigenvalues as follows
\[
\lambda_1(\Omega)=\inf_{\varphi\in W^{1,2}_0(\Omega)} \left\{\int_\Omega |\nabla\varphi|^2\,dx\, :\, \int_\Omega |\varphi|^2\,dx=1\right\},
\] 
and for $k\ge 2$
\[
\lambda_k(\Omega)=\inf_{\varphi\in W^{1,2}_0(\Omega)} \left\{\int_\Omega |\nabla\varphi|^2\,dx\, :\, \int_\Omega |\varphi|^2\,dx=1,\ \int_\Omega \varphi\,u_j\,dx=0, \ \text{for}\ j=1,\dots,k-1\right\}.
\]
Each eigenfunction $u_k$ is a minimizer of the corresponding problem.
See for example \cite[Chapter VI]{CH} or \cite[Chapter 4, Section 8]{BraBook} for these facts.
\par
We also recall that an alternative characterization of the eigenvalues is given by the well-known {\it Courant-Fischer formula}
\begin{equation}
\label{courant-fischer}
\lambda_k(\Omega)=\inf\left\{ \max_{\varphi\in K\cap \mathcal{S}_2(\Omega)} \int_\Omega|\nabla \varphi|^2\,dx\,:\, K \subseteq W^{1,2}_0(\Omega)\ \mbox{subspace with}\ \dim K=k \right\},
\end{equation}
where
\[
\mathcal{S}_2(\Omega):=\Big\{u\in L^2(\Omega)\, :\, \|u\|_{L^2(\Omega)}=1\Big\}.
\]
However, for a generic open set, it may happen that the embedding \eqref{mimmergo} fails to be compact. In this case, things become more complicated and intriguing: the spectrum of the Dirichlet-Laplacian stops being purely discrete (see \cite[Theorem 10.1.5]{BS}) and we have a non-empty {\it essential spectrum}, indicated by $\sigma_{\rm ess}(\Omega)$. 
\par
We recall that the latter can be characterized in terms of {\it singular constrained Palais-Smale sequences} (see for example \cite[Theorem 9.1.2]{BS} or \cite[Theorem 7.2]{HS}). Namely, we have that 
$\lambda\in \sigma_{\rm ess}(\Omega)$ if and only if there exists a sequence $\{\varphi_n\}_{n\in\mathbb{N}}\subseteq W^{1,2}_0(\Omega)$ such that
\begin{enumerate}
\item $\varphi_n\in\mathcal{S}_2(\Omega)$, for every $n\in\mathbb{N}$;
\vskip.2cm
\item $\lim\limits_{n\to\infty}\|\nabla \varphi_n\|_{L^2(\Omega)}^2=\lambda$;
\vskip.2cm
\item we have
\[
\lim_{n\to\infty} \|-\Delta \varphi_n-\lambda\,\varphi_n\|_{W^{-1,2}(\Omega)}=0;
\]
\item $\{\varphi_n\}_{n\in\mathbb{N}}$ weakly converges to $0$ in $L^2(\Omega)$.
\end{enumerate}
A celebrated result by Persson (see \cite[Theorem 2.1]{Pe}) gives a characterization of ``geometric''  flavour to the bottom of $\sigma_{\rm ess}(\Omega)$.
More precisely, if we introduce the {\it Poincar\'e constant ``at infinity''}
\[
\mathcal{E}(\Omega):=\sup_{R>0}\lambda_1\big(\Omega\setminus \overline{B_R}\big),
\]
then Persson's Theorem asserts that
\begin{equation}
\label{perso}
\inf \Big\{\lambda>0\, :\, \lambda\in\sigma_{\rm ess}(\Omega)\Big\}=\mathcal{E}(\Omega).
\end{equation}
Nevertheless, even when $\sigma_{\rm ess}(\Omega)\not=\emptyset$, it may happen that some discrete eigenvalues with finite multiplicities ``pop-up'' beneath the essential spectrum, i.e. there may exist some eigenvalues $\lambda<\mathcal{E}(\Omega)$. In this case, we have that these eigenvalues can still be characterized through the variational principles recalled above (see for example \cite[Theorems 10.2.1 \& 10.2.2]{BS}).

\par
We also recall that the corresponding eigenfunctions are strongly localized in space, in the sense that they exponentially ``fall-off'' at infinity.  In other words, they enjoy an estimate of the form
\begin{equation}
\label{falloff}
|u(x)|\le C\,e^{-\alpha\,|x|},\qquad \text{for}\ x\in\Omega,
\end{equation}
for suitable $C,\alpha>0$. One of the very first papers in the literature proving this kind of results is \cite{Sn} by \v{S}nol, for Schr\"odinger operators of the form $-\Delta+V$, with a non-negative potential $V$. We also wish to cite two papers \cite{Si, Si2} by Simon, which are connected with our main result presented below. 
\par
Nowadays, a result of the type \eqref{falloff} is a very particular instance of decay estimates that are usually named after Agmon, following his influential work \cite{Ag}. 
There, one can find a systematic study of decay estimates for eigenfunctions of $-\Delta+V$. Since then, it has become a classical subject in Spectral Theory. Many authors have contributed, by proving extensions, generalizations and refinements. Without any attempt of completeness, we mention the papers \cite{AH2OM} and \cite{CS}. An overview on Agmon estimates and their applications can be found in \cite{He1,He2}.
%

\subsection{Eigenvalues of the $p-$Laplacian}
The aim of this paper is to extend to the case of the Dirichlet $p-$Laplacian the previous analysis of 
\[
``\text{\it eigenvalues lying beneath the essential spectrum\,}''. 
\]
This sentence requires some precisions: at first, what it is intended by {\it eigenvalue} of the Dirichlet $p-$Laplacian on an open set $\Omega\subseteq\mathbb{R}^N$. From the point of view of Critical Point Theory, this is readily said: these eigenvalues can be understood as critical values of the functional
\[
\varphi\mapsto\int_\Omega |\nabla \varphi|^p\,dx,
\]
constrained to the ``sphere''
\[
\mathcal{S}_p(\Omega)=\Big\{u\in L^p(\Omega)\, :\, \|u\|_{L^p(\Omega)}=1\Big\}.
\]
The associated critical points are the eigenfunctions, accordingly.
Thus, in other words, eigenvalues of the Dirichlet $p-$Laplacian are those numbers $\lambda$ such that there exists $u\in W^{1,p}_0(\Omega)\setminus\{0\}$ weakly solving
\[
-\Delta_p u=\lambda\,|u|^{p-2}\,u,\qquad \text{in}\ \Omega,
\]
that is
\[
\int_\Omega \langle |\nabla u|^{p-2}\,\nabla u,\nabla\varphi\rangle\,dx=\lambda\,\int_\Omega |u|^{p-2}\,u\,\varphi,\qquad \text{for every}\ \varphi\in C^\infty_0(\Omega).
\]
Observe that the global infimum of the constrained Dirichlet integral
\[
\lambda_{1,p}(\Omega):=\inf_{\varphi\in W^{1,p}_0(\Omega)} \left\{\int_\Omega |\nabla\varphi|^p\,dx\, :\, \int_\Omega |\varphi|^p\,dx=1\right\},
\] 
corresponds to the {\it first eigenvalue} or {\it ground state energy}.
\par
On the other hand, what is intended by {\it essential spectrum} in this context is quite unclear. Nevertheless, by recalling Persson's result \eqref{perso} for the linear case, one might bravely guess that the {\it $L^p$ Poincar\'e constant at infinity}
\[
\mathcal{E}_{p}(\Omega):=\sup_{R>0}\lambda_{1,p}\big(\Omega\setminus \overline{B_R}\big),
\]
could be a valid surrogate for the infimum of the essential spectrum, in this nonlinear context. We are then lead to the following questions: is it possible to produce eigenvalues of the Dirichlet $p-$Laplacian below the threshold $\mathcal{E}_p(\Omega)$? Do the corresponding eigenfunctions enjoy a decay estimate similar to \eqref{falloff}? Attempting  to answer these questions, we will try to take very minimal assumptions on the open set $\Omega$.
\vskip.2cm\noindent
Let us start with the simpler case $\mathcal{E}_p(\Omega)=+\infty$: in this case, we have that the embedding
\begin{equation}
\label{mimmergop}
W^{1,p}_0(\Omega)\hookrightarrow L^p(\Omega),
\end{equation}
is compact, see for example Lemma \ref{lm:delcanto} below. In this scenario, it is well-known that we can produce a diverging sequence of positive eigenvalues by suitably mimicking the minmax Courant-Fischer procedure. This situation is very much studied in the literature, we just mention few classical references \cite{GP, Le, Szu}, the lecture notes \cite{Lin} and the PhD Thesis \cite{Fra}, where one can find many more references.
\par
We briefly recall the construction of this sequence of eigenvalues:
for every $k\in\mathbb{N}\setminus\{0\}$, we set
\[
\mathcal{W}_{k,p}(\Omega)=\Big\{K\subseteq \mathcal{S}_p(\Omega)\cap W^{1,p}_0(\Omega)\, :\, K\ \text{symmetric and compact},\, \gamma(K;W^{1,p}_0(\Omega))\ge k\Big\}.
\]
Here $\gamma(K;W^{1,p}_0(\Omega))$ indicates the {\it Krasnosel'ski\u{\i} genus} of $K$ in the norm topology of $W^{1,p}_0(\Omega)$, see \eqref{krasnodar} below for the definition.
With this class of compact sets at hand, we can define the {\it $k-$th minmax Ljusternik-Schnirelmann value} as follows
\begin{equation}
\label{minmaxLS}
\lambda_{k,p}^{\rm LS}(\Omega):=\inf_{K\in \mathcal{W}_{k,p}(\Omega)} \max_{\varphi\in K} \int_\Omega |\nabla \varphi|^p\,dx,\qquad k\in\mathbb{N}\setminus\{0\},
\end{equation}
and prove that this actually defines an eigenvalue, if the embedding \eqref{mimmergop} is compact. Let us notice that, for $k=1$, the previous definition boils down $\lambda_{1,p}(\Omega)$, i.e. it really gives the first eigenvalue. For this reason, in what follows we will omit the superscript ``${\rm LS}$'' in the case $k=1$. Moreover,  it is also possible to prove that, for  $k=2$, on a connected open set we have 
\[
\lambda_{2,p}^{\rm LS}(\Omega)=\min\Big\{\lambda>\lambda_{1,p}(\Omega)\,:\, \lambda\ \text{is an eigenvalue}\Big\},
\] 
see for instance \cite[Theorem 3.4]{JL}. 
\begin{rem}
In order to appreciate these apparently pedantic remarks about $k=1$ and $k=2$, the reader should keep in mind that it is still a major open problem to know whether $\{\lambda_{k,p}^{\rm LS}(\Omega)\}_{n\in\mathbb{N}\setminus\{0\}}$ exhaust the whole spectrum or not, at least in the case when \eqref{mimmergop} is compact. We also recall that it would be possible 
to define the previous minmax levels by means of another index $i$, in place of the Krasnosel'ski\u{\i} genus: for example, one could use the {\it $\mathbb{Z}_2-$cohomological index}
\cite{fadell_rabinowitz1977} or the {\it Ljusternik-Schnirelman Category} \cite[Chapter 2]{Str}.
\end{rem}
We then consider the case $\mathcal{E}_p(\Omega)<+\infty$: here the embedding \eqref{mimmergop} fails to be compact, see again Lemma \ref{lm:delcanto} below. Nevertheless, one could still define the minmax values \eqref{minmaxLS} and reasonably argue that whenever
\[
\lambda_{k,p}^{\rm LS}(\Omega)<\mathcal{E}_p(\Omega),
\]
we still have an eigenvalue, similarly to what happens in the linear case...
\subsection{Main result}
We will show that this is actually the case. Moreover, we will also get a \v{S}nol-Simon--type decay estimate for the relevant eigenfunction (we refer to Remark \ref{rem:SS} for the terminology). Namely, we prove the following
\begin{main}
Let $1<p<\infty$ and let $\Omega\subseteq\mathbb{R}^N$ be an open set. Let us suppose that there exists $k\in\mathbb{N}\setminus\{0\}$ such that
\begin{equation}\
\label{PL}
\lambda_{k,p}^{\rm LS}(\Omega)<\mathcal{E}_p(\Omega).
\end{equation}
Then, for every $\ell\in\{1,\dots,k\}$ the minmax value $\lambda_{\ell,p}^{\rm LS}(\Omega)$ is an eigenvalue of the Dirichlet $p-$Laplacian on $\Omega$. 
\par
Moreover, we have:
\begin{itemize}
\item if $\mathcal{E}_p(\Omega)<+\infty$,  for every $\ell\in\{1,\dots,k\}$ there exist two constants $C_\ell=C_\ell(N, p, \Omega)>0$ and $\alpha_\ell=\alpha_\ell(p,\Omega)>0$ such that for every eigenfunction $u_\ell\in W^{1,p}_0(\Omega)$ associated to $\lambda_{\ell,p}^{\rm LS}(\Omega)$, we have
\[
|u_\ell(x)|\le C_\ell\,\|u_\ell\|_{L^p(\Omega)}\,e^{-\alpha_\ell\,|x|}\,,\qquad \text{for every}\ x\in\Omega;
\]
\item if $\mathcal{E}_p(\Omega)=+\infty$, for every $\ell\in\mathbb{N}\setminus\{0\}$ and every $\alpha>0$ there exists a constant $C_\ell=C_\ell(N, p,\Omega,\alpha)>0$ such that 
\[
|u_\ell(x)|\le C_\ell\,\|u_\ell\|_{L^p(\Omega)}\,e^{-\alpha\,|x|}\,,\qquad \text{for every}\ x\in\Omega.
\]
\end{itemize}
\end{main}
The rest of this subsection is devoted to a couple of (detailed) comments on this result. We start with the exponential fall-off.
\begin{rem}[Fall-off]
\label{rem:SS}
As for the ``compact case'' $\mathcal{E}_p(\Omega)=+\infty$, our decay estimate represents an extension to the nonlinear case of \cite[Theorem 2]{Si2} by Simon, recalled above. The latter deals with Schr\"odinger operators $-\Delta+V$ on the whole $\mathbb{R}^N$, under the effect of a non-negative strongly confining potential $V$ (see Definition \ref{defi:scp} below).
\par
In the ``non-compact case'' $\mathcal{E}_p(\Omega)<+\infty$, our estimate is a nonlinear variation on the classical result by \v{S}nol, previously recalled. In particular, we have that
\[
\alpha_k\le \dots\le \alpha_1,
\]
with the notation above.
That is, the higher the eigenvalue, the slower the exponential decay. Moreover, these exponents deteriorate (i.e. they approach $0$) as the eigenvalues approach the spectral threshold $\mathcal{E}_p(\Omega)$: we refer to Proposition \ref{prop:lambdadecay} below, for this fact. 
This behaviour is in perfect accordance with \v{S}nol's result.
Indeed, in \cite[Theorem pag. 282]{Sn}, this estimate is proved for eigenfunctions of $-\Delta+V$, corresponding to an eigenvalue $\lambda$ below of the essential spectrum $\mathcal{E}_V$. Also in \cite{Sn}, the exponent $\alpha$ dictating the fall-off depends on the difference $\mathcal{E}_V-\lambda$ and it goes to $0$, as this difference goes to $0$. More precisely, \v{S}nol's exponent is given by 
\[
\alpha=\log\left(1+\frac{\mathcal{E}_V-\lambda}{K}\right),
\]
for constant $K$ depending on the (positive) potential $V$.
 For these reasons, in what follows we will refer to the estimate of the Main Theorem as {\it \v{S}nol-Simon--type decay estimate}.
\par
Finally, we wish to point out that our method of proof is different from those of both \cite{Si2} and \cite{Sn} and based on nonlinear techniques: we first obtain exponential decay in $L^p$ norm and then use an $L^\infty-L^p$ estimate ``localized at infinity'' to transform it into a pointwise decay estimate. The decay in $L^p$ norm is obtained by combining Caccioppoli and Poincar\'e inequalities, while the second estimate is a classical tool in the De Giorgi-Moser regularity theory for quasilinear equations.
\par 
It is worth mentioning that, once we obtain decay estimates on the eigenfunctions, it is possible to exploit the equation so to obtain some decay and integrability estimates on their gradients, as well (see Corollary \ref{coro:gradienti} below).
\end{rem}
We then briefly comment on how we prove the existence part of the Main Theorem.
\begin{rem}[Existence of eigenfunctions]
There are various ways to get existence of eigenfunctions. Of course, we can focus on the case $\mathcal{E}_p(\Omega)<+\infty$, in light of what we have previously recalled. In the Hilbertian setting (i.\,e. for $p=2$), a purely variational approach based on concentration-compactness arguments is provided by Smets, see \cite[Theorem 4.1 $\&$ Theorem 4.2]{Sm}, for a slightly different eigenvalue problem. His approach relies on the possibility to characterize higher eigenvalues as an infimum, possibly on the orthogonal complement of a finite dimensional vector space. This fact is true for $p=2$, but it has no counterpart for $p \neq 2$. For this reason, it is less clear whether this approach can be applied for the full range of $p$, see for instance \cite[Remark 4.4 (b)]{Sm}. 
\par
We briefly explain our route. We use a {\it spectral stability argument}, i.e. we prove that our minmax Ljusternik-Schnirelmann values can be approximated by the eigenvalues of a perturbed operator, having good compactness properties.
More precisely, we ``compactify'' the original eigenvalue problem by considering a new family of problems, each one obtained by adding a strongly confining potential $V_n$. For this family, we can prove existence of infinitely many ``perturbed'' eigenvalues, by the usual minmax procedure recalled above. Moreover, if $V_n$ tends to $0$ (in a suitable sense) as $n$ goes to $\infty$, we can prove that these eigenvalues converge to the original minmax Ljusternik-Schnirelmann values. Then, in order to prove that those lying below $\mathcal{E}_p(\Omega)$ actually define an eigenvalue, we aim to infer strong convergence in $L^p$ of the ``perturbed'' eigenfunctions. Here, we crucially exploit that these ``perturbed'' eigenfunctions enjoy a uniform \v{S}nol-Simon decay estimate, whenever they are associated to a ``perturbed'' eigenvalue which is converging to a $\lambda_{k,p}^{\rm LS}(\Omega)<\mathcal{E}_p(\Omega)$. The uniform fall-off in turn implies an $L^p$ equi-tightness property, which permits to appeal to the classical Riesz-Fr\'echet-Kolmogorov compactness criterion. In a nutshell, we obtain compactness through regularity. 
\par
In the language of Calculus of Variations, this is essentially a $\Gamma-$convergence--type argument in disguise, but used for higher critical points, other than the global minimum (see \cite{Brai,Dal} for the general theory of $\Gamma-$convergence). In doing this, we crucially borrow some ideas from the paper \cite{CD} by Champion and De Pascale.
\end{rem}

\subsection{Plan of the paper}
The paper starts with some preliminary materials contained in Section \ref{sec:2}. The main content of Section \ref{sec:3} is the \v{S}nol-Simon--type estimate of Proposition \ref{prop:lambdadecay}, which holds more generally for non-negative {\it subsolutions} of the eigenvalue equation. As a consequence, we also obtain some decay and integrability properties for their gradients. Section \ref{sec:4} is devoted to the ``perturbed'' eigenvalue problem described above, obtained by adding a strongly confining potential. Here, the technical Lemma \ref{lm:giggi} may be of some independent interest. The proof of the Main Theorem is contained in Section \ref{sec:5}, while in Section \ref{sec:6} we provide some examples of unbounded open sets to which our result applies. The paper ends with two appendices: in particular, in Appendix \ref{sec:B} we show that, for $p=2$, the minmax Ljusternik-Schnirelmann values below the essential spectrum coincide with the eigenvalues given by the Courant-Fischer formula.

\begin{ack}
We thank Bernard Helffer for some discussions on Agmon--type estimates and for providing us a copy of his papers \cite{He1, He2}.
We also thank Giovanni Franzina for some discussions on nonlinear eigenvalue problems and for a detailed reading of a preliminary draft of this paper.
\par
L.\,Briani and F.\, Prinari are both members of the {\it Gruppo Nazionale per l'Analisi Matematica, la Probabilit\`a
e le loro Applicazioni} (GNAMPA) of the Istituto Nazionale di Alta Matematica (INdAM). They both gratefully acknowledge the financial support of the project GNAMPA 2026  ``Problemi di ottimizzazione di forma in contesti anisotropi e non-locali" ({\tt  CUP E53C25002010001}).
\par
L.\,Brasco has been financially supported by the {\it Fondo di Ateneo per la Ricerca} FAR 2024 and the {\it Fondo per l'Incentivazione alla Ricerca Dipartimentale} FIRD 2025 of the University of Ferrara.
The research of L.\,Briani has been supported by the DFG through the Emmy Noether Programme (project number 509436910).    
\end{ack}

\section{Preliminaries}
\label{sec:2}

\subsection{Notation}
We will always denote by $N\in\mathbb{N}\setminus\{0\}$ the dimension of the ambient space. For $x_0\in \mathbb{R}^N$ and $R>0$, we set
\[
B_R(x_0)=\Big\{x\in\mathbb{R}^N\, :\, |x-x_0|<R\Big\}.
\] 
The writing $B_R$ will stand for the previous set, in the particular case where $x_0$ equals the origin.
\vskip.2cm\noindent
We recall the definition of {\it Krasnosel'ski\u{\i} genus}.
For every symmetric closed subset $A$ of a Banach space $X$, this is defined by
\begin{equation}
\label{krasnodar}
\gamma(A;X)=\inf\left\{k\in\mathbb{N}\, :\, \exists\ \text{a continuous odd map}\ f:A\to\mathbb{S}^{k-1}\right\},
\end{equation}
with the convention that $\gamma(A;X)=+\infty$, if no such an integer $k$ exists, and that $\gamma(\emptyset;X)=0$. We refer to \cite[Chapter II, Section 5]{Str} for more details on the genus.
\vskip.2cm\noindent
For an open set $\Omega\subseteq\mathbb{R}^N$ and an exponent $1<p<\infty$, we indicate by $W^{1,p}(\Omega)$ tha standard Sobolev space
\[
W^{1,p}(\Omega)=\Big\{\varphi\in L^p(\Omega)\, :\, \nabla \varphi\in L^p(\Omega)\Big\},
\]
endowed with the norm
\[
\|\varphi\|_{W^{1,p}(\Omega)}=\left(\|\varphi\|_{L^p(\Omega)}^p+\|\nabla\varphi\|_{L^p(\Omega)}^p\right)^\frac{1}{p},\qquad \text{for every}\ \varphi\in W^{1,p}(\Omega).
\]
By $W^{1,p}_0(\Omega)$ we mean the closure of $C^\infty_0(\Omega)$ in $W^{1,p}(\Omega)$. For $1<p<N$, we denote by $p^*$ the Sobolev exponent, defined by
\[
p^*=\frac{N\,p}{N-p}.
\]
\subsection{Some Poincar\'e inequalities}
We start with the following simple Poincar\'e--type inequality. The result is well-known, but since the case $p\ge N$ is often a bit overlooked in the literature, we prefer to give the details. The reader familiar with the subject can certainly skip it.
\begin{lm}
\label{lm:pinpon}
Let $1<p<\infty$ and let $E\subseteq\mathbb{R}^N$ be a measurable set with finite measure. Then, for every $\varphi\in C^\infty_0(\mathbb{R}^N)$ and every $\varepsilon>0$ we have:
\begin{itemize}
\item if $1<p<N$
\[
\int_E |\varphi|^p\,dx\le C\,|E|^\frac{p}{N}\,\int_{\mathbb{R}^N} |\nabla \varphi|^p\,dx;
\] 
\item if $p=N$
\[
\int_{E} |\varphi|^N\,dx\le \frac{C}{\varepsilon}\,|E|\,\int_{\mathbb{R}^N}|\nabla\varphi|^N\,dx+C\,\varepsilon\,\int_{\mathbb{R}^N} |\varphi|^N\,dx;
\]
\item if $p>N$
\[
\int_{E} |\varphi|^p\,dx\le \frac{C}{\varepsilon^\frac{p-N}{N}}\,|E|^\frac{p}{N}\,\int_{\mathbb{R}^N}|\nabla \varphi|^p+C\,\varepsilon\,\int_{\mathbb{R}^N} |\varphi|^p\,dx,
\]
\end{itemize}
for a constant $C=C(N,p)>0$.
\end{lm}
\begin{proof}
We first observe that for every $p<q\le \infty$, by H\"older's inequality we have\footnote{In the case $q=\infty$, it is intended that $p/q=0$.}
\begin{equation}
\label{holder}
\int_{E} |\varphi|^p\,dx\le |E|^{1-\frac{p}{q}}\,\|\varphi\|^p_{L^q(\mathbb{R}^N)}.
\end{equation}
We now proceed by distinguishing the three cases: $1<p<N$, $p=N$ and $p>N$.
\vskip.2cm\noindent
{\it Case $1<p<N$}. We choose $q=p^*$ in the previous inequality. Then, by Sobolev inequality we get
\[
\begin{split}
\int_{E} |\varphi|^p\,dx
&\le |E|^{1-\frac{p}{p^*}}\,\frac{1}{T_{N,p}}\,\int_{\mathbb{R}^N} |\nabla \varphi|^{p}\,dx,
\end{split}
\]
and we can conclude.
\vskip.2cm\noindent
{\it Case $p=N$}. By Ladyzhenskaya's inequality (see for example \cite[Theorem 3.6.4]{BraBook}), we have
\[
\|\varphi\|^N_{L^{2N}(\mathbb{R}^N)}\le\,\mathcal{L}_N\,\left(\int_{\mathbb{R}^N}|\nabla\varphi|^N\,dx\right)^\frac{1}{2}\,\left(\int_{\mathbb{R}^N} |\varphi|^N\,dx\right)^\frac{1}{2}.
\]
Thus, from \eqref{holder} with $q=2\,N$, we get
\[
\int_{E} |\varphi|^N\,dx\le |E|^\frac{1}{2}\,\mathcal{L}_N\,\left(\int_{\mathbb{R}^N}|\nabla\varphi|^N\,dx\right)^\frac{1}{2}\,\left(\int_{\mathbb{R}^N} |\varphi|^N\,dx\right)^\frac{1}{2}.
\]
By using Young's inequality, we get for every $\varepsilon>0$
\[
\int_{E} |\varphi|^N\,dx\le \frac{|E|\,\mathcal{L}_N}{2\,\varepsilon}\,\int_{\mathbb{R}^N}|\nabla\varphi|^N\,dx+\frac{\varepsilon\,\mathcal{L}_N}{2}\,\int_{\mathbb{R}^N} |\varphi|^N\,dx,
\]
as desired.
\vskip.2cm\noindent
{\it Case $p>N$}.
This is similar to the previous case. We can use Morrey's inequality (see for example \cite[Theorem 3.6.8]{BraBook})
\[
\|\varphi\|_{L^\infty(\mathbb{R}^N)}\le \mathcal{M}_N\,\|\varphi\|_{L^p(\mathbb{R}^N)}^{1-\frac{N}{p}}\,\|\nabla \varphi\|_{L^p(\mathbb{R}^N)}^\frac{N}{p}.
\]
Then, by \eqref{holder} with $q=\infty$, we get
\[
\int_{E} |\varphi|^p\,dx\le |E|\,\mathcal{M}^p_{N,p}\,\|\varphi\|_{L^p(\mathbb{R}^N)}^{p-N}\,\|\nabla \varphi\|_{L^p(\mathbb{R}^N)}^N.
\]
Again by Young's inequality, for every $\varepsilon>0$ we get
\[
\int_{E} |\varphi|^p\,dx\le \frac{N}{p}\,|E|^\frac{p}{N}\,\mathcal{M}^p_{N,p}\,\varepsilon^{-\frac{p-N}{N}}\,\int_{\mathbb{R}^N}|\nabla \varphi|^p+\frac{p-N}{p}\,\mathcal{M}_{N,p}^p\,\varepsilon\,\int_{\mathbb{R}^N} |\varphi|^p\,dx.
\]
This concludes the proof.
\end{proof}
The following weighted Poincar\'e-type inequality will be useful.
\begin{lm}
\label{lm:poincare}Let $V\in L^1_{\rm loc}(\mathbb{R}^N)$ be a non-negative function, such that there exist a constant $\alpha>0$ and a radius $R>0$ for which we have
\begin{equation}
\label{Vlouer}
V(x)\ge \alpha,\qquad \text{for a.\,e.}\ x\in\mathbb{R}^N\setminus B_R.
\end{equation}
Let $1<p<\infty$, then there exists a constant $C_{N,p,R,\alpha}>1$ such that we have 
\[
\frac{1}{C_{N,p,R,\alpha}}\,\int_{\mathbb{R}^N} |\varphi|^p\,dx\le \int_{\mathbb{R}^N} |\nabla \varphi|^p\,dx+\int_{\mathbb{R}^N} V\,|\varphi|^p\,dx,\qquad \text{for every}\ \varphi\in C^\infty_0(\mathbb{R}^N).
\]
\end{lm}
\begin{proof}
We take $\varphi\in C^\infty_0(\mathbb{R}^N)$ and decompose the integral over $\mathbb{R}^N$ as follows
\[
\int_{\mathbb{R}^N} |\varphi|^p\,dx=\int_{B_R} |\varphi|^p\,dx+\int_{\mathbb{R}^N\setminus B_R} |\varphi|^p\,dx.
\]
For the second integral, we simply use \eqref{Vlouer}, thus we get
\[
\int_{\mathbb{R}^N} |\varphi|^p\,dx\le \int_{B_R} |\varphi|^p\,dx+ \frac{1}{\alpha}\,\int_{\mathbb{R}^N\setminus B_R} V\,|\varphi|^p\,dx.
\]
In order to estimate the integral on $B_R$, we can appeal to Lemma \ref{lm:pinpon}. The case $1<p<N$ is immediate, we consider the remaining cases.
\vskip.2cm\noindent
{\it Case $p=N$}. By Lemma \ref{lm:pinpon}, we have for every $\varepsilon>0$
\[
\int_{\mathbb{R}^N} |\varphi|^N\,dx\le \frac{C}{\varepsilon}\,|B_R|\,\int_{\mathbb{R}^N}|\nabla\varphi|^N\,dx+C\,\varepsilon\,\int_{\mathbb{R}^N} |\varphi|^N\,dx+ \frac{1}{\alpha}\,\int_{\mathbb{R}^N\setminus B_R} V\,|\varphi|^N\,dx.
\]
By choosing $\varepsilon=1/(2C)$, we can absorb the $L^N$ norm of $\varphi$ and conclude.
\vskip.2cm\noindent
{\it Case $p>N$}. This is similar to the previous case. Still thanks to Lemma \ref{lm:pinpon} we have
\[
\int_{\mathbb{R}^N} |\varphi|^p\,dx\le \frac{C}{\varepsilon^\frac{p-N}{N}}\,|B_R|^\frac{p}{N}\,\int_{\mathbb{R}^N}|\nabla \varphi|^p+C\,\varepsilon\,\int_{\mathbb{R}^N} |\varphi|^p\,dx+ \frac{1}{\alpha}\,\int_{\mathbb{R}^N\setminus B_R} V\,|\varphi|^N\,dx.
\]
By choosing again $\varepsilon=1/(2C)$, we conclude.
\end{proof}
\begin{rem}
The precise form the constant $C_{N,p,R,\alpha}$ above is not very important for our scopes. We only stress that it exhibits the following natural behaviours
\[
\lim_{R\nearrow +\infty}C_{N,p,R,\alpha}=+\infty=\lim_{\alpha\searrow 0}C_{N,p,R,\alpha}.
\] 
\end{rem}
We then give a simple characterization of the compact embedding $W^{1,p}_0(\Omega)\hookrightarrow L^p(\Omega)$ in terms of $\mathcal{E}_p(\Omega)$. 
\begin{lm}
\label{lm:delcanto}
Let $1<p<\infty$ and let $\Omega\subseteq\mathbb{R}^N$ be an open set. Then
\[
W^{1,p}_0(\Omega)\hookrightarrow L^p(\Omega)\ \text{is compact}\qquad \Longleftrightarrow\qquad \mathcal{E}_p(\Omega)=+\infty.
\]
\end{lm}
\begin{proof}
The implication $\Longleftarrow$ is a direct consequence of \cite[Lemma A.1]{BraBriPri}, for example. For the converse implication, we prove that if $\mathcal{E}_p(\Omega)<+\infty$, then the embedding fails to be compact. 
By assumption, we have 
\[
\lambda_{1,p}(\Omega\setminus\overline{B_R})\le \mathcal{E}_p(\Omega)<+\infty,\qquad \text{for every}\ R>0.
\] 
For every $n\in\mathbb{N}\setminus\{0\}$, we take $\varphi_n\in C^\infty_0(\Omega\setminus \overline{B_n})\subseteq C^\infty_0(\Omega)\subseteq W^{1,p}_0(\Omega)$ such that
\[
\int_{\Omega\setminus \overline{B_n}} |\varphi_n|^p\,dx=1\qquad \mbox{ and }\qquad \lambda_{1,p}(\Omega\setminus\overline{B_n})\le \int_{\Omega\setminus\overline{B_n}} |\nabla \varphi_n|^p\,dx\le \lambda_{1,p}(\Omega\setminus\overline{B_n})+\frac{1}{n}.
\]
Such a function exists by definition of $\lambda_{1,p}(\Omega\setminus\overline{B_n})$. Thus, we have constructed a bounded sequence in $W^{1,p}_0(\Omega)$, with unit $L^p(\Omega)$ norm, which can not admit subsequences strongly converging in $L^p(\Omega)$. This gives that the embedding $W^{1,p}_0(\Omega)\hookrightarrow \mathcal{E}_p(\Omega)$ fails to be compact.
\end{proof}
At last, we present a simple result asserting that an open set supports the $L^p$ Poincar\'e inequality if and only if the Poincar\'e constant ``at infinity'' is positive. In light of the previous result, it is not restrictive to consider the case $\mathcal{E}_p<+\infty$ only.
\begin{lm}
\label{lm:minchione}
Let $1<p<\infty$ and let $\Omega\subseteq\mathbb{R}^N$ be an open set such that $\mathcal{E}_p(\Omega)<+\infty$. Then we have 
\[
\lambda_{1,p}(\Omega)>0\qquad \Longleftrightarrow\qquad \mathcal{E}_p(\Omega)>0.
\]
\end{lm}
\begin{proof}
The implication $\Longrightarrow$ follows from the fact that $\lambda_{1,p}(\Omega)\le \mathcal{E}_p(\Omega)$, which in turn is a consequence of the monotonicity of $\lambda_{1,p}$ with respect to the set inclusion.
\par 
In order to prove the converse implication, let us assume that $\mathcal{E}_p(\Omega)>0$ and show that $\Omega$ supports the $L^p$ Poincar\'e inequality. 
By definition and by monotonicity, this entails that there exists $R_\Omega>0$ such that 
\[
\lambda_{1,p}(\Omega\setminus \overline{B_R})\ge \frac{\mathcal{E}_p(\Omega)}{2},\qquad \mbox{ for every } R\ge R_\Omega.
\]
Let us take $\eta\in C^\infty_0(B_{2R_\Omega})$ such that
\[
0\le \eta\le 1,\qquad \eta\equiv 1 \mbox{ on } B_{R_\Omega},\qquad \|\nabla \eta\|_{L^\infty}\le \frac{C}{R_\Omega},
\]
for some dimensional constant $C>0$, 
and set $\zeta=1-\eta\in C^\infty(\mathbb{R}^N)$. For every $\varphi\in C^\infty_0(\Omega)$, we observe that 
\[
\eta\,\varphi\in W^{1,p}_0(\Omega\cap B_{2R_\Omega})\qquad \mbox{ and }\qquad \zeta\,\varphi\in W^{1,p}_0(\Omega\setminus\overline{B_{R_\Omega}}).
\] 
We then have
\[
\begin{split}
\int_{\Omega\cap B_{2R_\Omega}} |\varphi\,\eta|^p\,dx&\le \frac{1}{\lambda_{1,p}(\Omega\cap B_{2R_\Omega})}\,\int_{\Omega \cap B_{2R_\Omega}} |\nabla (\varphi\,\eta)|^p\,dx\\
&\le \frac{2^{p-1}}{\lambda_{1,p}(B_{2R_\Omega})}\,\int_{\Omega \cap B_{2R_\Omega}} |\nabla \varphi|^p\,dx+\frac{2^{p-1}\,C^p}{R_\Omega^p\,\lambda_{1,p}(B_{2R_\Omega})}\,\int_{B_{2R_\Omega}\setminus B_{R_\Omega}} |\varphi|^p\,dx.
\end{split}
\]
On the other hand, we also have 
\[
\begin{split}
\int_{\Omega\setminus B_{R_\Omega}} |\varphi\,\zeta|^p\,dx&\le \frac{1}{\lambda_{1,p}(\Omega\setminus\overline{B_{R_\Omega}})}\,\int_{\Omega \setminus B_{R_\Omega}} |\nabla (\varphi\,\zeta)|^p\,dx\\
&\le \frac{2^p}{\mathcal{E}_{p}(\Omega)}\,\int_{\Omega \setminus B_{R_\Omega}} |\nabla \varphi|^p\,dx+\frac{2^p\,C^p}{R_\Omega^p\,\mathcal{E}_{p}(\Omega)}\,\int_{B_{2R_\Omega}\setminus B_{R_\Omega}} |\varphi|^p\,dx.
\end{split}
\]
By summing the two estimates and using that
\[
\eta^p+\zeta^p\ge 2^{1-p}\,(\eta+\zeta)^p=2^{1-p}\qquad\text{and}\qquad \lambda_{1,p}(B_{2R_\Omega})=(2\,R_\Omega)^{-p}\,\lambda_{1,p}(B_1),
\] 
we get, after some standard manipulations
\begin{equation}
\label{intermesio}
\int_\Omega |\varphi|^p\,dx\le  C_{N,p}\,\left(R_\Omega^p+\frac{1}{\mathcal{E}_{p}(\Omega)}\right)\,\left[\int_\Omega |\nabla \varphi|^p\,dx+\frac{1}{R_\Omega^p}\,\int_{B_{2R_\Omega}} |\varphi|^p\,dx\right],
\end{equation}
for a suitable constant $C_{N,p}>0$.
\par
In order to handle the last term, we can appeal to Lemma \ref{lm:pinpon}. 
Indeed, in the case $1<p<N$, we directly get
\[
\frac{1}{R_\Omega^p}\,\int_{B_{2R_\Omega}} |\varphi|^p\,dx\le C\,\int_{\Omega} |\nabla \varphi|^{p}\,dx,
\]
which gives the claimed estimate.
Instead, for the case $p=N$, we get 
\[
\frac{1}{R_\Omega^N}\,\int_{B_{2R_\Omega}} |\varphi|^N\,dx\le \frac{C}{\varepsilon}\,\int_{\Omega}|\nabla\varphi|^N\,dx+\frac{C}{R_\Omega^N}\,\varepsilon\,\int_{\Omega} |\varphi|^N\,dx,
\]
which is valid for every $\varepsilon>0$, up to modify the constant $C$.
In conjunction with \eqref{intermesio}, this gives
\[
\begin{split}
\int_\Omega |\varphi|^N\,dx&\le  C_{N}\,\left(R_\Omega^N+\frac{1}{\mathcal{E}_{N}(\Omega)}\right)\,\left(1+\frac{1}{\varepsilon}\right)\,\int_\Omega |\nabla \varphi|^N\,dx\\
&+C_{N}\,\left(R_\Omega^N+\frac{1}{\mathcal{E}_{N}(\Omega)}\right)\,\frac{\varepsilon}{R_\Omega^N}\,\int_{\Omega} |\varphi|^N\,dx.
\end{split}
\]
By choosing 
\[
\varepsilon=\frac{R_\Omega^N}{2}\,\frac{1}{C_N}\,\frac{\mathcal{E}_N(\Omega)}{R_\Omega^N\,\mathcal{E}_N(\Omega)+1},
\]
we can absorb the $L^N$ norm on the right-hand side and prove the desired Poincar\'e inequality. Finally, for the case $p>N$ we can proceed similarly: we leave the details to the reader.
\end{proof}

\subsection{A weighted Sobolev space}
We now need a definition, in order to precise a specific class of potentials we wish to work with.
\begin{defi}
\label{defi:scp}
Let $V\in L^1_{\rm loc}(\mathbb{R}^N)$ be a non-negative function. We say that $V$ is a {\it strongly confining potential} if it has the following property: there exists an increasing function $\alpha:(0,+\infty)\to(0,+\infty)$ such that
\[
\lim_{R\to +\infty} \alpha(R)=+\infty,
\]
and for every $R>0$ we have
\[
V(x)\ge \alpha(R),\qquad \text{for a.e.}\ x\in \mathbb{R}^N\setminus B_R.
\]
\end{defi}
\begin{prop}
\label{prop:spazionostro}
Let $1<p<\infty$ and let $\Omega\subseteq\mathbb{R}^N$ be an open set. For a strongly confining potential $V\ge 0$, we define $\mathscr{D}^{1,p}_0(\Omega;V)$ to be the completion of $C^\infty_0(\Omega)$, with respect to the norm
\[
\|\varphi\|_{\mathscr{D}^{1,p}_0(\Omega;V)}:=\left(\int_\Omega |\nabla \varphi|^p\,dx+\int_\Omega V\,|\varphi|^p\,dx\right)^\frac{1}{p},\qquad \text{for every}\ \varphi\in C^\infty_0(\Omega).
\]
Then, this is a Banach space compactly embedded into $L^p(\Omega)$. Moreover, it can be identified with $W^{1,p}_0(\Omega;V)$, i.e. the closure of $C^\infty_0(\Omega)$ in the weigthed Sobolev space 
\[
W^{1,p}(\Omega;V)=\left\{\varphi\in W^{1,p}(\Omega)\, :\, \int_\Omega V\,|\varphi|^p\,dx<+\infty\right\},
\]
endowed with the norm
\[
\|\varphi\|_{W^{1,p}(\Omega;V)}:=\left(\int_\Omega |\nabla \varphi|^p\,dx+\int_\Omega (1+V)\,|\varphi|^p\,dx\right)^\frac{1}{p}.
\]
\end{prop}
\begin{proof}
By Lemma \ref{lm:poincare}, the two norms
\[
\|\varphi\|_{\mathscr{D}^{1,p}_0(\Omega;V)}\qquad \text{and}\qquad \|\varphi\|_{W^{1,p}(\Omega;V)},
\]
are equivalent on $C^\infty_0(\Omega)$. Thus, the two spaces $\mathscr{D}^{1,p}_0(\Omega;V)$ and $W^{1,p}_0(\Omega;V)$ can be identified in a canonical way.
Moreover, we note that for every $R>0$ it holds
\[
 \int_{\Omega\setminus B_{R}} |\varphi|^{p}\,dx\le  \frac{1}{\alpha(R)}\,\int_{\Omega\setminus B_{R}}V\,|\varphi|^{p}\,dx\le \frac{1}{\alpha(R)}\,\|\varphi\|_{\mathscr{D}^{1,p}_0(\Omega;V)}^p,\qquad \text{for every}\ \varphi \in W^{1,p}_0(\Omega;V).
\]
Thanks to the Riesz-Fr\'echet-Kolmogorov Theorem,  the previous  fact ensures that every bounded sequence $\{\varphi_n\}_{n\in\mathbb{N}}\subseteq W^{1,p}_0(\Omega;V) $ admits a subsequence strongly converging in $L^p(\Omega)$ to a function $\varphi$. Moreover, thanks to a standard argument (see for example the proof of \cite[Theorem 3.7.4]{BraBook}), we have that the space $W^{1,p}_0(\Omega;V) $ is also weakly closed and thus the limit function $\varphi$ still belongs to $W^{1,p}_0(\Omega;V)$. 
\end{proof}
The next simple technical result will be needed somewhere in the sequel.
\begin{lm}
\label{lm:staredentro}
Let $1<p<\infty$ and let $\Omega\subseteq\mathbb{R}^N$ be an open set. We take a strongly confining potential $V\ge 0$, then for every $\zeta\in C^\infty_0(\Omega)$ and $\varphi\in W^{1,p}_0(\Omega)\cap L^\infty(\Omega)$, we have
\[
\zeta\,\varphi\in W^{1,p}_0(\Omega;V).
\]
\end{lm}
\begin{proof}
By definition of $W^{1,p}_0(\Omega)$, there exists a sequence $\{\varphi_n\}_{n\in\mathbb{N}}\subseteq C^{\infty}_0(\Omega)$ that converges to $\varphi$ in $W^{1,p}(\Omega)$. Without loss of generality, we can further suppose that 
\begin{equation}\label{TCD}
\sup_{n\in\mathbb{N}}\|\varphi_n\|_{L^{\infty}(\Omega)}\le 2\, \|\varphi\|_{L^{\infty}(\Omega)} \qquad \text{and}\qquad  \lim_{n\to\infty}\varphi_n(x)=\varphi(x),\ \text{for a.\,e.}\ x\in\Omega. 
\end{equation}
Then, it is not difficult to verify that 
\[
\lim_{n\to\infty}\|\zeta\,\varphi_{n}-\zeta\,\varphi\|_{W^{1,p}(\Omega; V)}= 0.
\]
In particular, the fact that
\[
\lim_{n\to\infty}\int_{\Omega}V\,|\varphi_n-\varphi|^p\,|\zeta|^p\, dx=0
\]
follows from \eqref{TCD} and Lebesgue Dominated Convergence Theorem. Indeed, we have 
\[
V\,|\varphi_n-\varphi|^p\,|\zeta|^p\le 3^p\,\|\varphi\|^p_{L^\infty(\Omega)}\,V\,\zeta^p\in L^1(\Omega),
\]
thanks to the fact that $V\in L^1_{\rm loc}(\Omega)$ and $\zeta\in C^\infty_0(\Omega)$.
Since $\{\zeta\,\varphi_n\}_{n\in\mathbb{N}}\subseteq C^{\infty}_0(\Omega)$, the statement is proved.
\end{proof}

\section{Exponential fall-off of subsolutions}
\label{sec:3}

The \v{S}nol-Simon--type decay estimate of the Main Theorem will be a consequence of the following more general result, which holds for non-negative subsolutions.
\begin{prop}[Exponential decay at infinity] 
\label{prop:lambdadecay}
Let $1<p<\infty$ and let $\Omega\subseteq\mathbb{R}^N$ be an open set such that $\lambda_{1,p}(\Omega)>0$. For a constant $\lambda$ such that 
\[
\lambda_{1,p}(\Omega)\le \lambda<\mathcal E_p(\Omega),
\] 
we consider $v\in W^{1,p}_0(\Omega)$ a nonnegative subsolution of the equation
\begin{equation}\label{subsol}
-\Delta_p u = \lambda\, u^{p-1},\qquad \text{in}\ \Omega.
\end{equation}
Then:
\begin{itemize}
\item if $\mathcal{E}_p(\Omega)<+\infty$,  there exist two constants $\mathcal{C}=\mathcal{C}(N, p, \lambda, \Omega)>0$ and $\alpha=\alpha(p, \lambda, \Omega)>0$ such that  $v$
satisfies the following decay estimate 
\begin{equation}
\label{decayinfty}
0\le v(x)\le \mathcal{C}\,\|v\|_{L^p(\Omega)}\,e^{-\alpha\,|x|},\qquad \text{for a.\,e.}\ x\in\Omega.
\end{equation}
Moreover, the constant $\alpha$ decreasingly depends on $\lambda$ and we have 
\[
\lim_{\lambda\nearrow \mathcal{E}_p(\Omega)}\alpha=0;
\]
\item if $\mathcal{E}_p(\Omega)=+\infty$, for every $\alpha>0$ there exists a constant $\mathcal{C}=\mathcal{C}(N, p, \lambda, \Omega,\alpha)>0$ such that \eqref{decayinfty} holds.
\end{itemize}
\end{prop}
\begin{proof}
We adapt to the case of the $p-$Laplacian the proof of \cite[Theorem 5.1]{BiaBraOgn}. The proof will be slightly more elaborated, since we do not take any assumption on the ``geometry at infinity'' of the set $\Omega$, differently from the framework of \cite{BiaBraOgn}.
\par
We divide the proof in various parts, for ease of readability. 
\vskip.2cm\noindent
{\it Reduction to the $L^p$ norm.} We first claim that it is sufficient to prove that the desired exponential decay in $L^p$ norm, i.e. we just need to prove that:
\begin{itemize}
\item if $\mathcal{E}_p(\Omega)<+\infty$, there exist 
 $\alpha=\alpha(p,\lambda, \Omega)>0$ and a constant $C_1=C_1(p,\lambda,\Omega)>0$,  such that
\begin{equation}
\label{decayp}
\|v\|_{L^p(\Omega\setminus B_R)}\le C_1\, \|v\|_{L^p(\Omega)}\, e^{-\alpha\,R}, \qquad \text{for every } R>0;
\end{equation}
\item if $\mathcal{E}_p(\Omega)=+\infty$, for every $\alpha>0$ there exists a  constant $C_1=C_1(p,\lambda,\Omega,\alpha)>0$ such that \eqref{decayp} holds.
\end{itemize}
Indeed, assume that \eqref{decayp} holds. By using \cite[Lemma 2.3]{BraBriPri0} with $q=p$,  we have that $v\in L^\infty(\Omega)$ with \[
\|v\|_{L^\infty(\Omega)}\le C_2\,\lambda^\frac{N}{p^2}\,\|v\|_{L^p(\Omega)},
\]
for some $C_2=C_2(N,p)>0$. This in turn implies that 
\begin{equation}
\label{limitan}
\begin{split}
\|v\|_{L^\infty(B_2)}\le \Big(C_2\,e^{2\,\alpha}\,\lambda^\frac{N}{p^2}\Big)\,e^{-\alpha\, |x|}\,\|v\|_{L^p(\Omega)},\qquad \text{for}\ x\in B_2.
\end{split}
\end{equation}
Moreover, on account of \cite[Lemma 7.1]{BraBriPri0} with $q=p$,  we have that
\begin{equation}
\label{stimainfinito}
\|v\|_{L^\infty(\Omega\setminus B_{R+1})}\le C_3\,\left(1+\lambda \right)^\frac{N}{p^2}\,\|v\|_{L^p(\Omega\setminus B_{R})},\qquad \text{for every}\ R>0,
\end{equation}
for a constant $C_3=C_3(N,p)>0$.
 Then, from  \eqref{stimainfinito} and \eqref{decayp} with $R=n\in\mathbb{N}\setminus\{0\}$, we would  get  for almost every $x\in B_{n+2}\setminus B_{n+1}$
\[
\begin{split}
0\le v(x)\le \|v\|_{L^\infty(\Omega\setminus B_{n+1})}&\le C_3\,\left(1+\lambda\right)^\frac{N}{p^2}\,\|v\|_{L^p(\Omega\setminus B_n)}\\
&\le C_3\,C_1\,\left(1+\lambda\,\right)^\frac{N}{p^2}\, \|v\|_{L^p(\Omega)}\,e^{-\alpha\,n}\\
&\le \Big(C_3\,C_1\,  e^{2\,\alpha} \left(1+\lambda\,\right)^\frac{N}{p^2}\Big)\,\|v\|_{L^p(\Omega)}\,e^{-\alpha\,|x|}.
\end{split}
\]
The claimed estimate \eqref{decayinfty} would follow by combining the latter with \eqref{limitan} and setting
\[
\mathcal{C}:=\max\left\{C_2\,e^{2\,\alpha}\,\lambda^\frac{N}{p^2},\,C_3\,C_1\,  e^{2\,\alpha} \left(1+\lambda\,\right)^\frac{N}{p^2}\right\}.
\]
{\it Exponential decay in $L^p(\Omega)$: case $\mathcal{E}_p(\Omega)<+\infty$}. We are now going to show the validity of \eqref{decayp}. 
Let us start by noticing that, thanks to the choice of $\lambda$, we have
\[
\frac{\lambda}{\mathcal{E}_p(\Omega)}<1.
\]
We introduce the following function of one real variable
\[
\begin{array}{ccrcc}
h&:& \big[0,\min\left\{1,1/(p-1)\right\}\big]&\to&[0,1]\\
&&&&\\
&& t&\mapsto &\dfrac{1-t}{1+t}\,(1-(p-1)\,t),
\end{array}
\]
which is bijective and decreasing.
Thus, there exists $0<\varepsilon_{\Omega,\lambda}<1$ with the following property\footnote{Just observe that
\[
\frac{1}{2}\,\left(1+\frac{\lambda}{\mathcal{E}_p(\Omega)}\right)<1\qquad \text{and}\qquad \lim_{t\searrow 0} h(t)=1.
\]}
\begin{equation}\label{scelgoeps}
h(\varepsilon_{\Omega,\lambda})=\frac{1}{2}\,\left(1+\frac{\lambda}{\mathcal{E}_p(\Omega)}\right).
\end{equation}
We also notice that, by definition of $\mathcal{E}_p(\Omega)$, there exists $r_0=r_{0}(\Omega)>0$,  such that
\begin{equation}\label{scelgoR}
(1-\varepsilon_{\Omega,\lambda})\,\mathcal{E}_p(\Omega)\le \lambda_{1,p}(\Omega\setminus \overline{B_R}), \qquad \text{for every}\ R\ge r_0.
\end{equation}
For every $R>0$, we take the following radially symmetric Lipschitz cut-off function 
\begin{equation}\label{cutinfinito}
\eta(x)=\min\Big\{(|x|-R)_+,1\Big\}.
\end{equation}
Observe that
\[
0\le \eta\le 1,\qquad \eta\equiv 0 \ \text{on}\ B_R,\qquad \eta\equiv 1\ \text{on}\ \mathbb{R}^N\setminus B_{R+1},\qquad \|\nabla \eta\|_{L^\infty(\mathbb{R}^N)}= 1.
\]
We want to use this cut-off ``at infinity'' to obtain a suitable Caccioppoli inequality for $v$. 
More precisely, observe that by assumption $v$ satisfies the following the inequality 

\[
\int_{\Omega}\langle|\nabla v|^{p-2}\nabla v, \nabla \varphi\rangle\, dx\le \lambda\int_{\Omega}v^{p-1}\,\varphi\,dx,\qquad \text{for every}\ \varphi\in W^{1,p}_0(\Omega),\ \varphi\ge0.
\]

We use it with the test function  $\varphi=\eta^p\,v$, which is feasible. 
This gives
\begin{equation}
\label{caccio1}
p\int_{\Omega}\langle|\nabla v|^{p-2}\,\nabla v, \nabla \eta\rangle\, v\,\eta^{p-1}\,dx+ \int_{\Omega}|\nabla v|^{p}\,\eta^{p}\le  \lambda\,\int_{\Omega}v^{p}\,\eta^{p}\,dx.
\end{equation}
By using Young's inequality, we get for every $\delta>0$
\[
p\,\int_\Omega \langle |\nabla v|^{p-2}\,\nabla v,\nabla \eta\rangle\,\eta^{p-1}\,v\,dx\ge -(p-1)\,\delta\,\int_\Omega |\nabla v|^p\,\eta^p-\delta^{1-p}\,\int_\Omega |\nabla \eta|^p\,v^{p}\,dx.
\]

We plug this estimate into \eqref{caccio1}, so to readily deduce
\begin{equation}
\label{colbaffo}
(1-(p-1)\,\delta)\,\int_{\Omega}|\nabla v|^{p}\,\eta^{p}\,dx\le \delta^{1-p}\,\int_\Omega |\nabla \eta|^p\,v^{p}\,dx+ \lambda\,\int_{\Omega}v^{p}\,\eta^{p}\,dx.
\end{equation}
In particular, by choosing $\delta=\varepsilon_{\lambda,\Omega}$, we obtain
\begin{equation}\label{caccio2}
(1-(p-1)\,\varepsilon_{\Omega,\lambda})\,\int_{\Omega}|\nabla v|^{p}\,\eta^{p}\,dx\le \varepsilon_{\Omega,\lambda}^{1-p}\,\int_\Omega |\nabla \eta|^p\,v^{p}\,dx+ \lambda\,\int_{\Omega}v^{p}\,\eta^{p}\,dx.
\end{equation}
We now employ the elementary inequality given in Lemma \ref{lm:triangolare} below, with
\[
\varepsilon=\varepsilon_{\Omega,\lambda},\qquad a=\eta\,\nabla v\qquad \text{and}\qquad b=v\,\nabla \eta.
\] 
Thus, we have 
\[
\Big||\nabla(v\,\eta)|^p-|\nabla v|^p\,\eta^p\Big|\le \varepsilon_{\Omega,\lambda}\,|\nabla v|^p\,\eta^p+\varepsilon_{\Omega,\lambda}^{1-p}\,c_{p}\,v^p\,|\nabla \eta|^p.
\]
In particular, we can infer that
\[
\begin{split}
(1-(p-1)\,\varepsilon_{\Omega,\lambda})\,\int_{\Omega}|\nabla (v\,\eta)|^{p}\,dx&\le (1-(p-1)\,\varepsilon_{\Omega,\lambda})\,(1+\varepsilon_{\Omega,\lambda})\,\int_{\Omega}|\nabla v|^{p}\,\eta^{p}\,dx\\
&+(1-(p-1)\,\varepsilon_{\Omega,\lambda})\,\varepsilon_{\Omega,\lambda}^{1-p}\,c_{p}\,\int_\Omega v^p\,|\nabla \eta|^p\,dx\\
&\le (1+\varepsilon_{\Omega,\lambda})\,\lambda\,\int_{\Omega}v^{p}\,\eta^{p}\,dx\\
&+\varepsilon_{\Omega,\lambda}^{1-p}\,\Big[(1+\varepsilon_{\Omega,\lambda})+ (1-(p-1)\,\varepsilon_{\Omega,\lambda})\,c_{p}\Big]\,\int_\Omega |\nabla \eta|^p\,v^{p}\,dx.
\end{split}
\]
In the second inequality, we used \eqref{caccio2}. We now observe that 
\[
\int_{\Omega}|\nabla (v\,\eta)|^{p}\,dx\ge \lambda_{1,p}(\Omega\setminus\overline{B_R})\,\int_\Omega v^p\,\eta^p\,dx,
\]
thanks to the Poincar\'e inequality for $\eta\, u\in W^{1,p}_0(\Omega\setminus \overline{B_{R}})$.
By combining the last two estimates, we obtain
\begin{equation}
\label{caccio3}
\begin{split}
\lambda_{1,p}(\Omega\setminus\overline{B_R})\,\int_\Omega v^p\,\eta^p\,dx&\le \frac{1+\varepsilon_{\Omega,\lambda}}{(1-(p-1)\,\varepsilon_{\Omega,\lambda})}\,\lambda\,\int_{\Omega}v^{p}\,\eta^{p}\,dx\\ 
&+\varepsilon_{\Omega,\lambda}^{1-p}\,\left[\frac{(1+\varepsilon_{\Omega,\lambda})}{(1-(p-1)\,\varepsilon_{\Omega,\lambda})}+c_{p}\right]\,\int_\Omega |\nabla \eta|^p\,v^{p}\,dx.
\end{split}
\end{equation}
We can further estimate from below the leftmost term, by applying \eqref{scelgoR}, at least for large $R$. We thus obtain for every $R\ge r_0$
\[
\begin{split}
(1-\varepsilon_{\Omega,\lambda})\,\mathcal{E}_p(\Omega)\,\,\int_\Omega v^p\,\eta^p\,dx&\le\frac{1+\varepsilon_{\Omega,\lambda}}{(1-(p-1)\,\varepsilon_{\Omega,\lambda})}\,\lambda\,\int_{\Omega}v^{p}\,\eta^{p}\,dx\\ 
&+\varepsilon_{\Omega,\lambda}^{1-p}\,\left[\frac{(1+\varepsilon_{\Omega,\lambda})}{(1-(p-1)\,\varepsilon_{\Omega,\lambda})}+ c_{p}\right]\,\int_\Omega |\nabla \eta|^p\,v^{p}\,dx.
\end{split}
\]
We now use the properties of the cut-off function $\eta$: we get in particular
\[
\begin{split}
(1-\varepsilon_{\Omega,\lambda})\,\mathcal{E}_p(\Omega)\,\int_{\Omega\setminus B_{R+1}} v^p\,dx&\le\frac{1+\varepsilon_{\Omega,\lambda}}{(1-(p-1)\,\varepsilon_{\Omega,\lambda})}\,\lambda\,\int_{\Omega\setminus B_R}v^{p}\,dx\\ 
&+\varepsilon_{\Omega,\lambda}^{1-p}\,\left[\frac{(1+\varepsilon_{\Omega,\lambda})}{(1-(p-1)\,\varepsilon_{\Omega,\lambda})}+c_{p}\right]\,\int_{\Omega\cap(B_{R+1}\setminus B_R) } v^{p}\,dx.
\end{split}
\]
We further decompose
\[
\int_{\Omega\setminus B_R}v^{p}\,dx=\int_{\Omega\setminus B_{R+1}}v^{p}\,dx+\int_{\Omega\cap (B_{R+1}\setminus B_R)}v^{p}\,dx,
\]
thus with simple manipulations we arrive at
\[
\begin{split}
\left[\left(\frac{1-\varepsilon_{\Omega,\lambda}}{1+\varepsilon_{\Omega,\lambda}}\right)\right.&(1-(p-1)\varepsilon_{\Omega,\lambda})\,\mathcal{E}_p(\Omega)-\lambda \Big]\,\int_{\Omega\setminus B_{R+1}}v^p\,dx\ \\
&\le \left[ \varepsilon_{\Omega,\lambda}^{1-p}\,\left(1+ \frac{(1-(p-1)\varepsilon_{\Omega,\lambda})}{1+\varepsilon_{\Omega,\lambda}}\,c_{p}\right)+ \lambda\right]\,\int_{\Omega\cap (B_{R+1}\setminus B_R)}v^{p}\,dx.
\end{split}
\]
We polish a bit the previous estimate, by labeling the relevant constant. First, we observe that 
\[
\left(\frac{1-\varepsilon_{\Omega,\lambda}}{1+\varepsilon_{\Omega,\lambda}}\right)\,(1-(p-1)\varepsilon_{\Omega,\lambda})\,\mathcal{E}_p(\Omega)-\lambda=\frac{\mathcal{E}_p(\Omega)-\lambda}{2},
\]
by virtue of \eqref{scelgoeps}.
Then we set
\[
C_4=C_4(\Omega, p,\lambda)
:= \frac{2}{\mathcal{E}_p(\Omega)-\lambda}\,\left[\varepsilon_{\Omega,\lambda}^{1-p}\,\left(1+\frac{(1-(p-1)\varepsilon_{\Omega,\lambda})}{1+\varepsilon_{\Omega,\lambda}}\,c_{p}\right)+ \lambda\right].
\]
With this notation, we have obtained
\begin{equation}
\label{daquiuguale}
\int_{\Omega\setminus B_{R+1}}v^p\,dx\le C_4\int_{\Omega\cap (B_{R+1}\setminus B_R)}v^{p}\,dx.
\end{equation}
From the previous estimate, we can now easily deduce the desired exponential decay in the $L^p$ norm of $v$. 
It suffices to set
\[
A(R): = \int_{\Omega\setminus B_R}v^p\,dx,
\]
and observe that \eqref{daquiuguale} can be equivalently rewritten as
\[
A(R+1)\le\left(\frac{C_4}{1+C_{4}}\right)\,A(R),\qquad \text{for every}\ R\ge r_0.
\]
In particular for $R\ge r_0+1$ we have
\[
A(R)\le A(r_0+ \lfloor R- r_0\rfloor )\le \left(\frac{C_4}{1+C_4}\right)^{\lfloor R-r_0\rfloor}A(r_0)\le \left(\frac{C_4}{1+C_{4}}\right)^{ R-r_0 -1}\,\int_{\Omega}v^p\,dx.
\]
where we also used that
\[
A(r_0)= \int_{\Omega\setminus B_{r_0}}v^p\,dx\le \int_{\Omega}v^p\,dx.
\]
All in all, we have shown that
\[
\int_{\Omega\setminus B_{R}}v^p\,dx\le \left(\frac{C_4}{1+C_4}\right)^{R}\,\left(\frac{1+C_4}{C_4}\right)^{r_0+1}\,\int_{\Omega}v^p\,dx,\qquad  \text{for every}\ R\ge r_0+1.
\]
On the other hand, for $0<R<r_0+1$ we have
\[
\begin{split}
\int_{\Omega\setminus B_{R}}v^p\,dx\le \int_\Omega v^p\,dx&=\left(\frac{C_4}{1+C_4}\right)^{R}\, \left(\frac{1+C_4}{C_4}\right)^{R}\,\int_\Omega v^p\,dx\\
&\le \left(\frac{C_4}{1+C_4}\right)^{R}\, \left(\frac{1+C_4}{C_4}\right)^{r_0+1}\,\int_\Omega v^p\,dx.
\end{split}
\]
Thus, by setting
\[
C_1:=\left(1+\frac{1}{C_4}\right)^\frac{r_0+1}{p}\qquad \text{and}\qquad \alpha:=\frac{1}{p}\,\log\left(1+\frac{1}{C_4}\right),
\]
we get \eqref{decayp}. 
\par We conclude this part by proving the monotonicity property of the positive exponent $\alpha$. Indeed, by recalling the definition of the constant $C_4$, the exponent $\alpha$ is given by
\[
\alpha=\frac{1}{p}\,\log\left(1+\frac{1}{C_4}\right)=\frac{1}{p}\,\log\left\{1+\frac{\mathcal{E}_p(\Omega)-\lambda}{2}\,\frac{1}{\left[\varepsilon_{\Omega,\lambda}^{1-p}\,\left(1+\dfrac{(1-(p-1)\varepsilon_{\Omega,\lambda})}{1+\varepsilon_{\Omega,\lambda}}\,c_{p}\right)+ \lambda\right]}\right\}.
\]
In turn, by \eqref{scelgoeps} the constant $\varepsilon_{\Omega,\lambda}$ is given by 
\[
\varepsilon_{\Omega,\lambda}=h^{-1}\left(\frac{1}{2}\,\left(1+\frac{\lambda}{\mathcal{E}_p(\Omega)}\right)\right),
\]
 Thus, by recalling that $h$ (and thus $h^{-1}$, as well) is decreasing, we get that $\lambda\mapsto \varepsilon_{\Omega,\lambda}$ is decreasing, with 
\[
\lim_{\lambda\nearrow \mathcal{E}_p(\Omega)} \varepsilon_{\Omega,\lambda}=\lim_{\lambda\nearrow \mathcal{E}_p(\Omega)}h^{-1}\left(\frac{1}{2}\,\left(1+\frac{\lambda}{\mathcal{E}_p(\Omega)}\right)\right)=h^{-1}(1)=0.
\]
Accordingly, it is not too complicated to see that $\alpha$ decreasingly depends on $\lambda$ and 
\[
\lim_{\lambda\nearrow \mathcal{E}_p(\Omega)} \alpha=0,
\]
as claimed.
\vskip.2cm\noindent
{\it Exponential decay in $L^p(\Omega)$: case $\mathcal{E}_p(\Omega)=+\infty$}. It is enough to reproduce the previous arguments, by replacing in every occurrence $\mathcal{E}_p(\Omega)$ by any fixed number $\mathcal{E}$ strictly greater than $\lambda$. For example, one could take $\mathcal{E}=\vartheta\,\lambda$, where $\vartheta\ge 2$ is a fixed constant. Accordingly, we choose the constant $\varepsilon_\vartheta$ so to satisfy
\begin{equation}\label{scelgoeps2}
h(\varepsilon_\vartheta)=\frac{1+\vartheta}{2\,\vartheta},
\end{equation}
in place of \eqref{scelgoeps}.
Observe that actually such a choice is independent of both $\Omega$ and $\lambda$, this time. The proof then runs exactly as before, we leave the details to the reader. We just observe that we get this time
\[
\alpha=\frac{1}{p}\,\log\left\{1+\frac{(\vartheta-1)\,\lambda}{2}\,\frac{1}{\left[\varepsilon_\vartheta^{1-p}\,\left(1+\dfrac{(1-(p-1)\varepsilon_\vartheta)}{1+\varepsilon_\vartheta}\,c_{p}\right)+ \lambda\right]}\right\}.
\]
Since $\lambda\ge \lambda_{1,p}(\Omega)$, we have
\[
\alpha\ge \frac{1}{p}\,\log\left\{1+\frac{(\vartheta-1)\,\lambda_{1,p}(\Omega)}{2}\,\frac{1}{\left[\varepsilon_\vartheta^{1-p}\,\left(1+\dfrac{(1-(p-1)\varepsilon_\vartheta}{1+\varepsilon_\vartheta}\,c_{p}\right)+ \lambda_{1,p}(\Omega)\right]}\right\}.
\]
Moreover,  by construction,  using that $h^{-1}$ is decreasing and   $\vartheta\ge 2$, we have that 
\[
\varepsilon_{\vartheta}=h^{-1}\left(\frac{1+\vartheta}{2\,\vartheta}\right)\ge h^{-1}\left(\frac{3}{4}\right)=:\overline{\varepsilon}.
\] We can thus infer
\[
\alpha\ge \frac{1}{p}\,\log\left\{1+\frac{(\vartheta-1)\,\lambda_{1,p}(\Omega)}{2}\,\frac{1}{\left[\overline{\varepsilon}^{1-p}\,\left(1+\dfrac{c_p}{1+\overline{\varepsilon}}\right)+ \lambda_{1,p}(\Omega)\right]}\right\}=:\frac{1}{p}\,\log\left(1+\frac{\vartheta-1}{K_{\Omega,p}}\right).
\]
Hence, we can obtain the claimed exponential decay with the last exponent, which is independent of $\lambda$ and can be made arbitrarily large, since $\vartheta\ge 2$ is arbitrary.
\end{proof}
\begin{rem}
In the case $\mathcal{E}_p(\Omega)=+\infty$, as one may easily expect, taking larger and larger exponents $\alpha$ leads to a quantitative loss in the estimate. In other words, the constant $\mathcal{C}$ blows-up, accordingly.
Indeed, a closer inspection of the proof shows that in this case we can take
\[
\mathcal{C}=\max\left\{C_2\,e^{2\,\alpha}\,\lambda^\frac{N}{p^2},\,C_3\,\left(1+\frac{1}{C_4}\right)^\frac{r_0+1}{p}\,  e^{2\,\alpha} \left(1+\lambda\,\right)^\frac{N}{p^2}\right\},
\]
thus it diverges to $+\infty$, as $\alpha$ goes to $+\infty$.
We also observe that
\[
\lim_{\alpha\to+\infty}C_4= 0,
\]
as well as possibly
\[
\lim_{\alpha\to +\infty}r_0=+\infty.
\]
\end{rem}
As a consequence of the previous result, we can also obtain some decay and integrability properties of the gradient. This is the content of the following result, that we record for completeness.
\begin{coro}
\label{coro:gradienti}
Let $1<p<\infty$ and let $\Omega\subseteq\mathbb{R}^N$ be an open set such that $\lambda_{1,p}(\Omega)>0$. For a constant $\lambda$ such that 
\[
\lambda_{1,p}(\Omega)\le \lambda<\mathcal E_p(\Omega),
\] 
we consider $v\in W^{1,p}_0(\Omega)$ a nonnegative subsolution of the equation \eqref{subsol}.
Then\footnote{We use the notation $\mathcal{C}$ for the same constant of Proposition \ref{prop:lambdadecay}.}:
\begin{itemize}
\item if $\mathcal{E}_p(\Omega)<+\infty$, there exists $M=M(p,\lambda,\Omega)>0$ such that 
 \begin{equation}
\label{decaygrad}
\|\nabla v\|_{L^p(\Omega\setminus B_{R+1})}\le  M\, e^{-\alpha\,R}\,\|v\|_{L^p(\Omega)}, \qquad \text{for every}\ R>0,
\end{equation}
where $\alpha$ is the same exponent as in Proposition \ref{prop:lambdadecay}. Moreover,  for every $0<\beta<\alpha$ there exists a constant  
\[
\widetilde{\mathcal{C}}= \widetilde{\mathcal{C}}(\mathcal{C}, p, \lambda, \beta, \alpha-\beta)>0,
\]
such that 
\begin{equation}
\label{decaygradiente}
\int_{\Omega}\left|\nabla v\right|^{p}\,e^{\beta p|x|}\,dx\le  \widetilde{\mathcal{C}}\,\|v\|^p_{L^{p}(\Omega)};
\end{equation}
\item if $\mathcal{E}_p(\Omega)=+\infty$, for every $\alpha>0$ there exists a constant $M=M(p, \lambda, \Omega,\alpha)>0$ such that \eqref{decaygrad} holds. Moreover,  for every $\beta>0$ there exists a constant  
\[
\widetilde{\mathcal{C}}= \widetilde{\mathcal{C}}(\mathcal{C}, p, \lambda, \beta)>0,
\]
such that \eqref{decaygradiente} holds.
\end{itemize}
In particular, in both cases we have that 
\[
\int_\Omega |\nabla v|^\gamma\,dx<+\infty,\qquad \text{for every}\ 0<\gamma\le p.
\]
\end{coro}
\begin{proof}
For simplicity, we consider the case $\mathcal{E}_p(\Omega)<+\infty$, only. The other case can be treated in the same way.
\par
In order to show \eqref{decaygrad}, we go back to the Caccioppoli inequality \eqref{colbaffo}. By choosing $\delta=1/(2\,(p-1))$, we get
\[
\frac{1}{2}\,\int_{\Omega}|\nabla v|^{p}\,\eta^{p}\,dx\le 2^{p-1}\,(p-1)^{p-1}\,\int_\Omega |\nabla \eta|^p\,v^{p}\,dx+ \lambda\,\int_{\Omega}v^{p}\,\eta^{p}\,dx.
\]
Thanks to the properties of the cut-off function $\eta$, this in turn gives
\[
\frac{1}{2}\,\int_{\Omega\setminus B_{R+1}}|\nabla v|^{p}\,dx\le \Big(2^{p-1}\,(p-1)^{p-1}+\lambda\Big)\,\int_{\Omega\setminus B_R} v^{p}\,dx.
\]
In order to conclude, we can apply the exponential decay of the $L^p$ norm proved in \eqref{decayp} on the right-hand side. This gives the claimed decay estimate \eqref{decaygrad}, with
\[
M:= (2^p\,(p-1)^{p-1}+2\,\lambda)^\frac{1}{p}\, C_1.
\] 
The proof of \eqref{decaygradiente} is similar, we simply amend the choice of the test function. We fix $R>0$ and consider a standard cut-off function $\eta\in C^{\infty}_0(B_{R+1})$ satisfying
\[
0\le \eta\le 1,\qquad \eta\equiv 1 \ \text{on}\ B_R,\qquad \|\nabla \eta\|_{L^\infty(\mathbb{R}^N)}\le  C,
\]
for some dimensional constant $C>0$. We take $0<\beta<\alpha$ as in the statement and use in the weak formulation the following test function
\[
\varphi= \eta^{p}\, e^{\beta p|x|}\, v.
\]
We obtain
\[
\begin{split}
\int_{\Omega}|\nabla v|^{p}\,\eta^{p}\, e^{p\beta|x|}\,dx \le \lambda \int_{\Omega}v^{p}\, \eta^{p}\, e^{p\beta|x|}\,dx&-p\,\beta\, \int_{\Omega}\left\langle |\nabla v|^{p-2}\,\nabla v ,  \frac{x}{|x|}\right\rangle\, v\,  \eta^p\, e^{p\beta|x|}\,dx\\
 &-p\,\int_{\Omega}\langle |\nabla v|^{p-2}\,\nabla v, \nabla \eta \rangle \, v\, \eta^{p-1}\, e^{p\beta|x|}\,dx.
\end{split}
\]
We now use Young's inequality in the right-hand side: for every $\delta>0$, we then obtain
\[
\begin{split}
\int_{\Omega}|\nabla u|^{p}\,\eta^{p}\,e^{p\beta|x|}\,dx&\le \lambda \int_{\Omega}v^{p}\,\eta^{p}\,e^{p\beta|x|}\,dx\\
& +\delta\, (p-1)\beta \int_{\Omega} |\nabla v|^{p}\,\eta^{p}\,e^{p\beta|x|}\,dx+ \delta^{1-p}\,\beta\,\int_{\Omega}v^{p}\,\eta^{p}\,e^{p\beta|x|}\,dx \\
&+\delta\, (p-1)\int_{\Omega}|\nabla v|^{p}\, \eta^{p}\,e^{p\beta|x|}\,dx+ \delta^{1-p}\, \int_{\Omega}  v^{p}\,|\nabla \eta|^{p} \, e^{p\beta|x|}\,dx.
\end{split}
\]
We can choose in particular
\[
\delta=\frac{1}{2\,(p-1)\,(1+\beta)},
\]
and absorb in the left-hand side the terms containing $\nabla v$. By exploiting also the properties of $\eta$ and using some simple algebraic manipulations, this leads to
\[
\begin{split}
\int_{\Omega\cap B_R}|\nabla v|^{p}\,e^{p\beta|x|}\,dx\le C_p\,(\lambda +1+\beta^p)\,\int_{\Omega\cap B_{R+1}} v^{p}\, e^{p\beta|x|}\,dx.
\end{split}
\]
We can now appeal to the decay estimate \eqref{decayinfty}, so to obtain
\[
\begin{split}
\int_{\Omega\cap B_R}|\nabla v|^{p}\,e^{p\beta|x|}\,dx\le C_p\,(\lambda +1+\beta^p)\, \mathcal{C}^p\,\left(\int_{\mathbb{R}^N}  \, e^{-p(\alpha-\beta)|x|}\,dx\right) \|v\|^p_{L^{p}(\Omega)}=: \widetilde{\mathcal{C}}\, \|v\|^p_{L^{p}(\Omega)}.
\end{split}
\]
By taking the limit as $R$ goes to $+\infty$ and using the Monotone Convergence Theorem, we get \eqref{decaygradiente} as desired.
\par
Finally, for every $0<\gamma< p$,  H\"older's inequality provides 
\[
\int_{\Omega}|\nabla v|^\gamma\,dx\le \left(\int_{\Omega}|\nabla v|^{p}\, e^{p\beta |x|}\,dx\right)^{\frac{\gamma}{p}}\left(\int_{\Omega}e^{-\frac{p\beta}{p-\gamma}|x|} \, dx\right)^\frac{p-\gamma}{p}<+\infty.
\]
In light of \eqref{decaygradiente}, this shows the desired integrability.
\end{proof}

\section{Perturbed eigenvalue problems}
\label{sec:4}

We will need the following crucial technical result. We state it in a slightly larger generality than that really needed for our scopes. The proof is largely inspired by that of \cite[Theorem 3.3]{CD}.
\begin{lm}
\label{lm:giggi}
Let $1<p<\infty$ and let $\Omega\subseteq\mathbb{R}^N$ be an open set such that $\lambda_{1,p}(\Omega)>0$. Let $X(\Omega)$ be a Banach space with the following properties:
\begin{itemize}
\item $C^\infty_0(\Omega)\subseteq X(\Omega)$;
\vskip.2cm
\item $X(\Omega)\subseteq W^{1,p}_0(\Omega)$ with continuous inclusion, i.e. there exists $C>0$ such that
\[
\|\nabla \varphi\|_{L^p(\Omega)}\le C\, \|\varphi\|_{X(\Omega)},\qquad \text{for every}\ \varphi\in X(\Omega).
\]
\end{itemize}
We set
\[
\mathcal{X}_{k,p}(\Omega)=\Big\{K\subseteq \mathcal{S}_p(\Omega)\cap X(\Omega)\, :\, K\ \text{symmetric and compact},\, \gamma(K;X(\Omega))\ge k\Big\},
\]
and
\[
\mathcal{W}_{k,p}(\Omega)=\Big\{K\subseteq \mathcal{S}_p(\Omega)\cap W^{1,p}_0(\Omega)\, :\, K\ \text{symmetric and compact},\, \gamma(K;W^{1,p}_0(\Omega))\ge k\Big\}.
\]
Then, for every $K\in \mathcal{W}_{k,p}(\Omega)$ and every $\delta>0$, there exists $K_\delta\in \mathcal{X}_{k,p}(\Omega)$ such that
\[
\max_{\varphi\in K_\delta} \int_\Omega |\nabla \varphi|^p\,dx\le \max_{\varphi\in K} \int_\Omega |\nabla \varphi|^p\,dx+\delta.
\]
\end{lm}
\begin{proof}
We first notice that the assumption $\lambda_{1,p}(\Omega)>0$ ensures that 
\begin{equation}
\label{equivale}
\varphi\mapsto \|\nabla \varphi\|_{L^p(\Omega)},
\end{equation}
is an equivalent norm on $W^{1,p}_0(\Omega)$.
We take $K$ and $\delta>0$, as in the statement. We set for simplicity
\[
\mathrm{M}:=\max_{\varphi\in K} \int_\Omega |\nabla \varphi|^p\,dx,
\]
which is a positive real number and define
\begin{equation}
\label{costantina}
c_\Omega:=\left(\frac{1}{\lambda_{1,p}(\Omega)}\right)^\frac{1}{p}\,\frac{5}{2}.
\end{equation}
At last, we fix a radius $r=r(\mathrm{M},c_\Omega,\delta)>0$ such that
\begin{equation}
\label{chosen1}
r<\frac{1}{2\,c_\Omega},
\end{equation}
and
\begin{equation}
\label{chosen2}
(1-c_\Omega\,r)^{-p}\,\left[1+\frac{r}{2\,\mathrm{M}^\frac{1}{p}}\right]^p\le 1+\frac{\delta}{\mathrm{M}}.
\end{equation}
Observe that such a choice is feasible, since for $r\searrow 0$ the left-hand side in \eqref{chosen2} converges to $1$, while the right-hand side is strictly larger than $1$. 
\par
The reason for this choice will be clear in a while, let us now proceed with the construction of the set $K_\delta$ declared in the statement.
\vskip.2cm\noindent
We start by observing that $K$ is compact, thus we can find a finite number of functions $\{\varphi_1,\dots,\varphi_\ell\}\subseteq K$ such that 
\[
K\subseteq \bigcup_{i=1}^\ell \left\{\varphi\in W^{1,p}_0(\Omega)\, :\, \|\nabla \varphi-\nabla\varphi_i\|_{L^p(\Omega)}<\frac{r}{2}\right\},
\]
where $r>0$ is the radius previously introduced.
In particular, each $\varphi_i$ belongs to $\mathcal{S}_p(\Omega)$, thus it has unit $L^p(\Omega)$ norm. By density of $C^\infty_0(\Omega)$ in $W^{1,p}_0(\Omega)$, we can choose $\{\psi_1,\dots,\psi_\ell\}\subseteq C^\infty_0(\Omega)$ such that
\[
\|\nabla \psi_i-\nabla\varphi_i\|_{L^p(\Omega)}<\frac{r}{2},\qquad \text{for}\ i\in\{1,\dots,\ell\}.
\]
This implies that for every $i\in\{1,\dots,\ell\}$ we have
\[
\left\{\varphi\in W^{1,p}_0(\Omega)\, :\, \|\nabla \varphi-\nabla\varphi_i\|_{L^p(\Omega)}<\frac{r}{2}\right\}\subseteq \left\{\varphi\in W^{1,p}_0(\Omega)\, :\, \|\nabla \varphi-\nabla\psi_i\|_{L^p(\Omega)}<r\right\}
\]
and thus we get
\begin{equation}
\label{annalaura}
K\subseteq \bigcup_{i=1}^\ell \left\{\varphi\in W^{1,p}_0(\Omega)\, :\, \|\nabla \varphi-\nabla\psi_i\|_{L^p(\Omega)}<r\right\},
\end{equation}
as well. We define 
\[
F:=\{\pm \psi_1,\dots,\pm\psi_\ell\}^{\rm ch},
\]
i.e. this is the {\it convex hull} of the finite family $\{\pm \psi_1,\dots,\pm\psi_\ell\}$. Observe that, since $C^\infty_0(\Omega)\subseteq X(\Omega)$, the previous set $F$ is in particular a symmetric convex closed subset of $X(\Omega)$, which is compact, as well. By indicating with $\mathrm{Proj}_F$ the projection operator from $W^{1,p}_0(\Omega)$ to $F$, we then set
\[
\widetilde{K}=\mathrm{Proj}_F(K).
\]
Here the projection is intended with respect to the equivalent norm \eqref{equivale}
on $W^{1,p}_0(\Omega)$.
\par
Let us verify that the elements of $\widetilde{K}$ are uniformly detached from the origin. More precisely, we claim that
\begin{equation}
\label{louer}
\|\psi\|_{L^p(\Omega)}\ge 1-c_\Omega\,r,\qquad \text{for every}\ \psi\in\widetilde{K},
\end{equation}
where $c_\Omega$ is the constant defined in \eqref{costantina}. Indeed, let us take $\psi\in \widetilde{K}$, this implies that $\psi=\mathrm{Proj}_F(\varphi)$ for some $\varphi\in K$. By property \eqref{annalaura}, there exists $i\in\{1,\dots,\ell\}$ such that 
\[
\|\nabla \varphi-\nabla \psi_i\|_{L^p(\Omega)}<r.
\]
Thus, from the triangle inequality and the properties of projection operators (recall that $\psi_i\in F$), we have
\[
\begin{split}
\|\psi\|_{L^p(\Omega)}=\|\mathrm{Proj}_F(\varphi)\|_{L^p(\Omega)}&\ge \|\varphi\|_{L^p(\Omega)}-\|\mathrm{Proj}_F(\varphi)-\varphi\|_{L^p(\Omega)}\\
&\ge \|\varphi\|_{L^p(\Omega)}-\left(\frac{1}{\lambda_{1,p}(\Omega)}\right)^\frac{1}{p}\,\|\nabla\mathrm{Proj}_F(\varphi)-\nabla\varphi\|_{L^p(\Omega)}\\
&\ge \|\varphi\|_{L^p(\Omega)}-\left(\frac{1}{\lambda_{1,p}(\Omega)}\right)^\frac{1}{p}\,\|\nabla\psi_i-\nabla\varphi\|_{L^p(\Omega)}\\
&\ge \|\varphi_i\|_{L^p(\Omega)}-\|\varphi_i-\varphi\|_{L^p(\Omega)}-\left(\frac{1}{\lambda_{1,p}(\Omega)}\right)^\frac{1}{p}\,\|\nabla\psi_i-\nabla\varphi\|_{L^p(\Omega)}\\
&\ge 1-\left(\frac{1}{\lambda_{1,p}(\Omega)}\right)^\frac{1}{p}\,\left(\|\nabla\varphi_i-\nabla\varphi\|_{L^p(\Omega)}+\|\nabla\psi_i-\nabla\varphi\|_{L^p(\Omega)}\right)\\
&\ge 1-\left(\frac{1}{\lambda_{1,p}(\Omega)}\right)^\frac{1}{p}\,\left(\|\nabla\varphi_i-\nabla\psi_i\|_{L^p(\Omega)}+2\,\|\nabla\psi_i-\nabla\varphi\|_{L^p(\Omega)}\right)\\
&\ge 1-\left(\frac{1}{\lambda_{1,p}(\Omega)}\right)^\frac{1}{p}\,\frac{5\,r}{2}.
\end{split}
\]
We remark that we have also used Poincar\'e inequality twice.
In conclusion, we have obtained \eqref{louer}, as claimed.
\par
Now, we observe that the set $\widetilde{K}$ is symmetric and, by virtue of the continuity of the projection operator, is compact  with respect to the norm topology of $W^{1,p}_0(\Omega)$. Actually, it is compact with respect to the norm topology of $X(\Omega)$, as well. Indeed, we have $\widetilde{K}\subseteq F\subseteq X(\Omega)$ and by construction $F$ in turn is contained in a finite dimensional subspace, where all the norms are equivalent. This is enough to justify our assertion. 
\par
Nevertheless, the set $\widetilde{K}$ is still not the one we are looking for: indeed, it is not a subset of $\mathcal{S}_p(\Omega)$. Thus we have to renormalize its elements. To this aim, we consider the odd continuous functional $T:W^{1,p}_0(\Omega)\to W^{1,p}_0(\Omega)$ given by
\[
T(\psi)=\left\{\begin{array}{rl}
\dfrac{\psi}{\|\psi\|_{L^p(\Omega)}},& \text{if}\ \|\psi\|_{L^p(\Omega)}\ge \dfrac{1}{2},\\
&\\
2\,\psi,& \text{otherwise}.
\end{array}
\right.
\]
Observe that the restriction of $T$ to $F$ takes values in $X(\Omega)$ and it is continuous with respect to the norm topology of $X(\Omega)$.
On account of \eqref{louer} and the choice \eqref{chosen1}, we clearly have 
\begin{equation}
\label{nellasfera}
T(\psi)=\dfrac{\psi}{\|\psi\|_{L^p(\Omega)}}\in \mathcal{S}_p(\Omega)\cap X(\Omega),\qquad \text{for every}\ \psi\in \widetilde{K}.
\end{equation}
Finally, we are ready to define
\[
K_\delta=T(\widetilde{K})=T\circ \mathrm{Proj}_F(K).
\]
From what previously observed, we get that $K_\delta$ is  a subset of $\mathcal{S}_p(\Omega)\cap X(\Omega)$ which is compact in the norm topology of $X(\Omega)$ and  symmetric. 
\par
To show that $K_\delta\in\mathcal{X}_{k,p}(\Omega)$, we have also to estimate its genus, still with respect to the norm topology of $X(\Omega)$.
From \cite[Chapter II, Proposition 5.4]{Str}, by observing that $T\circ \mathrm{Proj}_F$ can be regarded as a continuous odd functional from $W^{1,p}_0(\Omega)$ to itself and by using that $K\in\mathcal{W}_{k,p}(\Omega)$, we deduce that 
\[
\gamma(K_\delta;W^{1,p}_0(\Omega))\ge \gamma(K;W^{1,p}_0(\Omega))\ge k.
\] 
Since, by appealing to Lemma \ref{lm:generoditoma}, we have
\[
\gamma(K_\delta;X(\Omega))=\gamma(K_\delta;W^{1,p}_0(\Omega)),
\]
we get $K_\delta\in \mathcal{X}_{k,p}(\Omega)$, as desired.
\par
We are left with estimating the $p-$Dirichlet integral for functions belonging to $K_\delta$: by construction, for every $\eta\in K_\delta$, we have 
\[
\eta=\frac{\psi}{\|\psi\|_{L^p(\Omega)}},\qquad \text{for some}\ \psi\in\widetilde{K}.
\]
Thus, by \eqref{louer}
\begin{equation}
\label{PL1}
\int_\Omega |\nabla \eta|^p\,dx=\frac{\displaystyle\int_\Omega |\nabla \psi|^p\,dx}{\displaystyle\int_\Omega |\psi|^p\,dx}\le (1-c_\Omega\,r)^{-p}\,\int_\Omega |\nabla\psi|^p\,dx.
\end{equation}
By using that $\psi\in \widetilde{K}=\mathrm{Proj}_F(K)$, we have that $\psi$ can be written as a convex combinations of the functions $\{\pm\psi_1,\dots,\pm\psi_\ell\}$. Thus, by convexity and evenness of the $p-$Dirichlet integral, we obtain
\begin{equation}
\label{PL2}
\int_\Omega |\nabla\psi|^p\,dx\le \max_{\{1,\dots,\ell\}} \int_\Omega |\nabla\psi_i|^p\,dx.
\end{equation}
In order to conclude, we observe that for every $i\in\{1,\dots,\ell\}$ we have by triangle inequality
\[
\begin{split}
\int_\Omega |\nabla\psi_i|^p\,dx&\le \left[\left(\int_\Omega |\nabla\varphi_i|^p\,dx\right)^\frac{1}{p}+\|\nabla \psi_i-\nabla\varphi_i\|_{L^p(\Omega)}\right]^p\\
&\le \left[\mathrm{M}^\frac{1}{p}+\frac{r}{2}\right]^p=\mathrm{M}\,\left[1+\frac{r}{2\,\mathrm{M}^\frac{1}{p}}\right]^p.
\end{split}
\]
It is now time for the weird choice \eqref{chosen2} to come into play: this assures that the last quantity is not larger than $(\mathrm{M}+\delta)\,(1-c_\Omega\,r)^{p}$. Thus, we get
\[
\int_\Omega |\nabla\psi_i|^p\,dx\le (\mathrm{M}+\delta)\,(1-c_\Omega\,r)^{p},\qquad \text{for every}\ i\in\{1,\dots,\ell\}.
\]
By using this information in \eqref{PL2} and going back to \eqref{PL1}, we finally get
\[
\int_\Omega |\nabla \eta|^p\,dx\le \mathrm{M}+\delta,\qquad \text{for every}\ \eta\in K_\delta.
\]
This eventually establishes the claimed assertion.
\end{proof}
For a strongly confining potential $V$, we extend to the weighted space $W^{1,p}_0(\Omega;V)$ (recall Proposition \ref{prop:spazionostro}) the definition of $\mathcal{W}_{k,p}(\Omega)$. Namely, we introduce
\[
\mathcal{W}_{k,p}(\Omega;V)=\Big\{K\subseteq \mathcal{S}_p(\Omega)\cap W^{1,p}_0(\Omega;V)\, :\, K\ \text{symmetric and compact},\, \gamma(K;W^{1,p}_0(\Omega;V))\ge k\Big\}.
\]
Thanks to the previous result, we can prove an approximation results for the minmax values \eqref{minmaxLS}. This is a generalization of \cite[Lemma 4.1]{BraBriPri0} to higher critical values, in the case $q=p$. It is the main result of this section.
\begin{prop}
\label{prop:approssimazione}
Let $1<p<\infty$ and let $\Omega\subseteq\mathbb{R}^N$ be an open set such that $\lambda_{1,p}(\Omega)>0$. Finally, let $V\ge 0$ be a strongly confining potential and let $\{\varepsilon_n\}_{n\in\mathbb{N}}$ be a decreasing infinitesimal sequence of positive numbers.
For every $n\in\mathbb{N}$ and $k\in\mathbb{N}\setminus\{0\}$, we define 
\[
\lambda_{k,p}^{\rm LS}(\Omega;V_n):=\inf_{K\in \mathcal{W}_{k,p}(\Omega;V)} \max_{\varphi\in K}\, \mathcal{G}_n(\varphi),
\]
where 
\[
\mathcal{G}_n(\varphi)=\int_{\Omega}|\nabla \varphi|^{p}\,dx +\int_{\Omega}V_n\,|\varphi|^{p}\,dx\qquad \text{and}\qquad V_n(x)=\varepsilon_n\,V(x).
\]
Then, we have
\begin{equation}
\label{vaiaffa}
\lim_{n\to\infty}\lambda_{k,p}^{\rm LS}(\Omega;V_n)=\lambda_{k,p}^{\rm LS}(\Omega).
\end{equation}
Moreover, the quantity $\lambda_{k,p}^{\rm LS}(\Omega;V_n)$ is a critical value for the functional $\mathcal{G}_n$ restricted to the manifold $\mathcal{S}_{p}(\Omega)\cap W^{1,p}_0(\Omega;V)$, i.e. there exists a function $u_{k,n}\in \mathcal{S}_p(\Omega)\cap W^{1,p}_0(\Omega;V)$  such that 
\[
\begin{split}
\int_\Omega \langle |\nabla u_{k,n}|^{p-2}\,\nabla u_{k,n},\nabla\varphi\rangle\,dx&+\int_\Omega V_n\,|u_{k,n}|^{p-2}\,u_{k,n}\,\varphi\,dx\\
&=\lambda_{k,p}^{\rm LS}(\Omega;V_n)\,\int_\Omega |u_{k,n}|^{p-2}\,u_{k,n}\,\varphi\,dx,\quad \text{for every}\ \varphi\in C^\infty_0(\Omega).
\end{split}
\]
Finally, each function $u_{k,n}$ is such that 
\[
u_{k,n}\in L^\infty(\Omega)\qquad \text{and}\qquad \|u_{k,n}\|_{L^\infty(\Omega)}\le C_{N,p}\,\Big(\lambda_{k,p}^{\rm LS}(\Omega;V_n)\Big)^\frac{N}{p^2}.
\]

\end{prop}
\begin{proof}
We first observe that 
\[
\mathcal{G}_n(\varphi)\le \max\{1,\varepsilon_0\}\left(\int_{\Omega}|\nabla \varphi|^{p}\,dx +\int_{\Omega}V\,|\varphi|^{p}\,dx\right),\qquad \text{for every}\ \varphi\in W^{1,p}_0(\Omega;V).
\]
Thus, the functional $\mathcal{G}_n$ is finite and $C^1$ on the Banach space $W^{1,p}_0(\Omega;V)$.
We divide the proof in three parts, according to the claim we want to prove.
\vskip.2cm\noindent
{\it Proof of \eqref{vaiaffa}}. We observe that, thanks to Lemma \ref{lm:generoditoma}, we have $\mathcal{W}_{k,p}(\Omega;V)\subseteq \mathcal{W}_{k,p}(\Omega)$. Moreover, 
\[
\mathcal{G}_n(\varphi)\ge \int_\Omega |\nabla \varphi|^p,\qquad \text{for every}\ \varphi\in W^{1,p}_0(\Omega;V).
\]
These immediately imply that 
\[
\lambda_{k,p}^{\rm LS}(\Omega;V_n)\ge \lambda_{k,p}^{\rm LS}(\Omega)\qquad \text{and thus}\qquad \liminf_{n\to\infty}\lambda_{k,p}^{\rm LS}(\Omega;V_n)\ge\lambda_{k,p}^{\rm LS}(\Omega).
\]
In order to prove the reverse inequality, we take $K\in\mathcal{W}_{k,p}(\Omega)$ and $\delta>0$. We apply Lemma \ref{lm:giggi} with $X(\Omega)=W^{1,p}_0(\Omega;V)$: then there exists $K_\delta\in\mathcal{W}_{k,p}(\Omega;V)$ such that
\[
\max_{\varphi\in K_\delta} \int_\Omega |\nabla \varphi|^p\,dx\le \max_{\varphi\in K} \int_\Omega |\nabla \varphi|^p\,dx+\delta.
\]
We also observe that 
\[
\varphi\mapsto \int_\Omega V\,|\varphi|^p\,dx,
\]
is a continuous functional on $W^{1,p}_0(\Omega;V)$. Thus, by compactness (with respect to the norm topology) of $K_\delta$, we have that there exists a constant $C_{\delta, K_\delta}>0$ such that
\[
\int_\Omega V\,|\varphi|^p\,dx\le C_{\delta, K_\delta},\qquad \text{for every}\ \varphi\in K_\delta.
\]
In particular, we have that
\[
\int_\Omega V_n\,|\varphi|^p\,dx\le \varepsilon_n\,C_{\delta, K_\delta},\qquad \text{for every}\ \varphi\in K_\delta.
\]
In conclusion, we get
\[
\lambda^{\rm LS}_{k,p}(\Omega;V_n)\le \max_{\varphi\in K_\delta} \mathcal{G}_n(\varphi)\le \max_{\varphi\in K} \int_\Omega |\nabla \varphi|^p\,dx+\delta+\varepsilon_n\,C_{\delta, K_\delta}.
\]
This implies that 
\[
\limsup_{n\to\infty}\lambda^{\rm LS}_{k,p}(\Omega;V_n)\le \max_{\varphi\in K} \int_\Omega |\nabla \varphi|^p\,dx+\delta.
\]
By taking the infimum over $K\in\mathcal{W}_{k,p}(\Omega)$, this in turn implies 
\[
\limsup_{n\to\infty}\lambda^{\rm LS}_{k,p}(\Omega;V_n)\le \lambda^{\rm LS}_{k,p}(\Omega)+\delta.
\]
This is enough to get the reverse inequality, thanks to the arbitrariness of $\delta>0$.
\vskip.2cm\noindent
{\it Proof of the criticality}. We wish to appeal to classical minmax theorems, for example \cite[Corollary 4.1]{Szu}. We fix $n\in \mathbb N$, we just have to check that the restriction of $\mathcal{G}_n$ to the manifold $\mathcal{S}_p(\Omega)\cap W^{1,p}_0(\Omega;V)$ satisfies the Palais-Smale condition, at every level $\lambda\ge \lambda_{1,p}^{\rm LS}(\Omega;V_n)$. In other words, we have to show that if $\{\varphi_m\}_{m\in\mathbb{N}}$ is such that
\begin{equation}
\label{PS0}
\{\varphi_m\}_{m\in\mathbb{N}}\subseteq \mathcal{S}_p(\Omega)\cap W^{1,p}_0(\Omega;V),
\end{equation}
\begin{equation}
\label{PS1}
\lim_{m\to\infty}\mathcal{G}_n(\varphi_m)=\lambda,
\end{equation}
and
\begin{equation}
\label{PS2}
\lim_{m\to\infty} \left(\sup_{\|\psi\|_{W^{1,p}_0(\Omega;V)}\le1} \bigg|\delta\mathcal{G}_n[\varphi_m](\psi)-\lambda\,p\,\int_\Omega |\varphi_m|^{p-2}\,\varphi_m\,\psi\,dx\bigg|\right)=0,
\end{equation}
then there exists $\overline{\varphi}\in\mathcal{S}_p(\Omega)\cap W^{1,p}_0(\Omega;V)$ such that
\[
\lim_{m\to\infty} \|\varphi_m-\overline{\varphi}\|_{W^{1,p}(\Omega;V)}=0,
\]
up to a subsequence. Here we denoted by $\delta\mathcal{G}_n[\varphi_m]$ the first variation of $\mathcal{G}_n$ computed at $\varphi_m$, i.e.
\[
\delta\mathcal{G}_n[\varphi_m](\psi)=p\int_\Omega \langle |\nabla \varphi_m|^{p-2}\,\nabla \varphi_m,\nabla\psi\rangle\,dx+p\int_\Omega V_n\,|\varphi_m|^{p-2}\,\varphi_m\,\psi\,dx,\quad \text{for every}\ \psi\in W^{1,p}_0(\Omega;V).
\]
We notice that the sequence $\{\varphi_{m}\}_{m\in\mathbb{N}}$ is bounded in $W^{1,p}_0(\Omega; V)$: indeed
\[
\begin{split}
 \|\varphi_m\|_{W^{1,p}_0(\Omega;V)}\le \Big(1+\max\{1,\varepsilon_n^{-1}\}\,\mathcal{G}_n(\varphi_m)\Big)^\frac{1}{p},
\end{split}
\]
and the last term is uniformly bounded in $m$, thanks to \eqref{PS0} and \eqref{PS1}.
 Therefore, up to a subsequence, it converges to some $\overline{\varphi}\in W^{1,p}_0(\Omega; V)$, weakly in $W^{1,p}_0(\Omega; V)$ and strongly in $L^p(\Omega)$, thanks to Proposition \ref{prop:spazionostro}. This already ensures that
\[
\overline{\varphi}\in \mathcal{S}_p(\Omega)\cap W^{1,p}_0(\Omega;V).
\] 
We can now proceed with a standard argument. 
We set for simplicity
\[
\theta_m:=\sup_{\|\psi\|_{W^{1,p}_0(\Omega;V)}\le1} \bigg|\delta\mathcal{G}_n[\varphi_m](\psi)-\lambda\,p\,\int_\Omega |\varphi_m|^{p-2}\,\varphi_m\,\psi\,dx\bigg|.
\]
On account of \eqref{PS2}, this is an infinitesimal sequence of positive numbers. 
In particular, we have
\begin{equation}\label{testo1}
\bigg|\delta\mathcal{G}_n[\varphi_m](\varphi_m-\overline{\varphi})\bigg|  \le \theta_m\,\|\varphi_m-\overline{\varphi}\|_{W^{1,p}_0(\Omega;V)}+ \lambda\,p\,\left|\int_{\Omega}|\varphi_m|^{p-2}\,\varphi_m\, (\varphi_m-\overline{\varphi})\,dx \right|.
\end{equation}
We note that, by H\"older's inequality and \eqref{PS0}, it holds 
\begin{equation}\label{1pezzo}
\left|\int_{\Omega}|\varphi_m|^{p-2}\varphi_m (\varphi_m-\overline{\varphi})\,dx\right|\le\|\varphi_m-\overline{\varphi}\|_{L^p(\Omega)}.
\end{equation}
Since the sequence $\{\varphi_m\}_{m\in \mathbb N}$ is bounded in $W^{1,p}_0(\Omega;V)$ and is strongly converging in $L^{p}(\Omega)$, from \eqref{testo1} and  \eqref{1pezzo}  we get
\[
\lim_{m\to\infty}\bigg|\delta\mathcal{G}_n[\varphi_m](\varphi_m-\overline{\varphi})\bigg|=0.
\] 
On the other hand, by the weak convergence, we have
\[
\lim_{m\to\infty}\bigg|\delta\mathcal{G}_n[\overline{\varphi}](\varphi_m-\overline{\varphi})\bigg|=0,
\]
as well.
By putting them together, we obtain
\[
\lim_{m\to\infty}\bigg|\delta\mathcal{G}_n[\varphi_m](\varphi_m-\overline{\varphi})-\delta\mathcal{G}_n[\overline{\varphi}](\varphi_m-\overline{\varphi})\bigg|= 0.
\]
By recalling the definition of $\mathcal{G}_n$, this permits to infer that
\begin{equation}
\label{testo3/2}
0 =\lim_{m\to\infty}\int_{\Omega}\langle|\nabla \varphi_m|^{p-2}\nabla \varphi_m-|\nabla\overline{\varphi}|^{p-2}\,\nabla\overline{\varphi},\nabla \varphi_m-\nabla \overline{\varphi}\rangle\,dx,
\end{equation}
and 
\begin{equation}
0 =\lim_{m\to\infty}\int_{\Omega} V\,\left(|\varphi_m|^{p-2}\,\varphi_m-|\overline{\varphi}|^{p-2}\,\overline{\varphi}\right)\,(\varphi_m-\overline{\varphi})\,dx.
\end{equation}
We can now appeal to standard monotonicity inequalities for the power function $z\mapsto|z|^p$, in order to get the desired strong convergence in $W^{1,p}_0(\Omega;V)$.
\vskip.2cm\noindent
{\it $L^\infty$ bound}. The fact that each $u_{k,n}$ is globally bounded, together with the claimed $L^\infty$ estimate, can be obtained by observing that $|u_{k,n}|$ is a non-negative weak subsolution of
\[
-\Delta_p v+V_n\,v^{p-1}= \lambda_{k,p}^{\rm LS}(\Omega;V_n)\,v^{p-1},\qquad \text{in}\ \Omega.
\]
By using that $V_n\ge 0$, this in particular implies that $|u_{k,n}|\in W^{1,p}_0(\Omega;V)\subseteq W^{1,p}_0(\Omega)$ weakly satisfies
\[
-\Delta_p |u_{k,n}|\le \lambda_{k,p}^{\rm LS}(\Omega;V_n)\,|u_{k,n}|^{p-1},\qquad \text{in}\ \Omega.
\]
We can thus conclude by appealing to \cite[Lemma 2.3]{BraBriPri0}.
\end{proof}
\begin{rem}
By construction, we have that $V_n\le V_0$ and thus it is not difficult to see that
\[
\lambda_{k,p}^{\rm LS}(\Omega;V_n)\le \lambda_{k,p}^{\rm LS}(\Omega;V_0),\qquad \text{for every}\ n\in\mathbb{N}.
\]
Therefore, from the previous $L^\infty$ bound, we get in particular
\begin{equation}
\label{uniformM}
\|u_{k,n}\|_{L^\infty(\Omega)}\le C_{N,p}\,\Big(\lambda_{k,p}^{\rm LS}(\Omega;V_0)\Big)^\frac{N}{p^2}=:M_k,
\end{equation}
which is uniform with respect to $n$.
\end{rem}
\begin{coro}
\label{coro:uniformdecay}
Let $1<p<\infty$ and let $\Omega\subseteq\mathbb{R}^N$ be an open set such that  $\mathcal{E}_p(\Omega)<+\infty$. Let us suppose that there exists $k\in\mathbb{N}\setminus\{0\}$ such that
\[
\lambda_{k,p}^{\rm LS}(\Omega)<\mathcal{E}_p(\Omega).
\]
With the notation of Proposition \ref{prop:approssimazione}, for every $1\le \ell\le k$ there exists an index $n_\ell\in\mathbb{N}$ and two positive constants $C_\ell$ and $\alpha_\ell$ independent of $n$ such that
\begin{equation}
\label{unifdecay}
|u_{\ell,n}(x)|\le C_\ell\, e^{-\alpha_\ell\,|x|},\qquad \text{for a.\,e.}\ x\in\Omega\ \text{and for every}\ n\ge n_\ell. 
\end{equation}
\end{coro}
\begin{proof}
We first observe that, with our standing hypotheses,  $\mathcal{E}_p(\Omega)>0$. In particular, Lemma \ref{lm:minchione} ensures that $\lambda_{1,p}(\Omega)>0$.  The assumption on $\lambda_{k,p}^{\rm LS}(\Omega)$, together with the property \eqref{vaiaffa}, implies that for every $\ell\in\{1,\dots,k\}$ there exists $n_\ell\in\mathbb{N}$ and a positive number $\lambda_\ell$ with $\lambda_{\ell,p}^{\rm LS}(\Omega)< \lambda_\ell<\mathcal{E}_p(\Omega)$, such that
\[
\lambda_{\ell,p}^{\rm LS}(\Omega;V_n)< \lambda_\ell,\qquad \text{for every}\ n\ge n_\ell,
\]
as well. In particular, we get that $|u_{\ell,n}|\in W^{1,p}_0(\Omega)$ is a nonnegative subsolution of the equation
\[
-\Delta_p u = \lambda_\ell\, u^{p-1},\qquad \text{in}\ \Omega, 
\]
 for every $n\ge n_\ell$. Taking into account that  $u_{\ell,n}$ has unit $L^p(\Omega)$ norm,  by applying Proposition \ref{prop:lambdadecay} we get that  there exist $C=C(N,p,\lambda_\ell, \Omega)$ and  $\alpha=\alpha(p,\lambda_\ell, \Omega)$ such that   $u_{\ell,n}$ satisfies the uniform decay estimate
\[
|u_{\ell,n}(x)|\le C_\ell\,e^{-\alpha_\ell\,|x|},\qquad \text{for a.\,e.}\ x\in\Omega \ \text{and for every}\ n\ge n_\ell.
\]
This concludes the proof.
\end{proof}

\section{Proof of the Main Theorem}
\label{sec:5}
As previously observed,  the assumption \eqref{PL} ensures that  $\lambda_{1,p}(\Omega)>0$.
We then divide the proof in two parts: we first prove that each minmax value below the threshold $\mathcal{E}_p(\Omega)$ is an eigenvalue. Then, we show how to get exponential decay at infinity of the relevant eigenfunctions.
\vskip.2cm\noindent 
{\it Step 1: low eigenvalues}. We can suppose that $\mathcal{E}_p(\Omega)<+\infty$, otherwise we already know that every $\lambda_{k,p}^{\rm LS}(\Omega)$ defines an eigenvalue, as recalled in the Introduction. We wish to show that for every fixed index $1\le \ell\le k$ as in the statement, the sequence $\{u_{\ell,n}\}_{n\in\mathbb{N}}$ constructed in Proposition \ref{prop:approssimazione} converges to an eigenfunction for $\lambda_{k,p}^{\rm LS}(\Omega)$, up to a subsequence. We first observe that by construction we have
\begin{equation}
\label{fratocuggino}
\{u_{\ell,n}\}_{n\in\mathbb{N}}\subseteq \mathcal{S}_p(\Omega)\cap W^{1,p}_0(\Omega)\quad \text{and}\quad \int_\Omega |\nabla u_{\ell,n}|^p\,dx\le \lambda_{\ell,p}^{\rm LS}(\Omega;V_0),\qquad \text{for every}\ n\in\mathbb{N}.
\end{equation}
Thus, the sequence $\{u_{\ell,n}\}_{n\in\mathbb{N}}$ is bounded in $W^{1,p}_0(\Omega)$.  Consequently, it converges weakly in $W^{1,p}_0(\Omega)$ to some limit function $u_\ell\in  W^{1,p}(\Omega)$, up to a subsequence. We can further assure that this convergence is strong in $L^p(\Omega)$, thanks to Corollary \ref{coro:uniformdecay}, which in turn permits to apply the classical Riesz-Fr\'echet-Kolmogorov compactness result. In particular, we  get that $u_\ell\in\mathcal{S}_p(\Omega)$ and thus $u_\ell\not\equiv 0$. We claim that $u_\ell$ is the desired eigenfunction. In order to prove this fact, we will show that we can pass to the limit in the weak formulation of the equation
\begin{equation}
\label{autosalone}
-\Delta_p u_{\ell,n}+V_n\,|u_{\ell,n}|^{p-2}\,u_{\ell,n}= \lambda_{\ell,p}^{\rm LS}(\Omega;V_n)\,|u_{\ell,n}|^{p-2}\, u_{\ell,n},\qquad \text{in}\ \Omega.
\end{equation}
To this aim, let $R>0$ and $\zeta\in C^{\infty}_0(\mathbb{R}^N)$  a cut-off function such that
\[
0\le \zeta\le 1,\quad \zeta=1 \text{ on } B_{R},\quad \zeta=0\ \text{on}\ \mathbb{R}^{N}\setminus B_{2R}, \quad \|\nabla \zeta\|_{\infty}\le \frac{C}{R},
\]
for some dimensional constant $C>0$.
We note that, thanks to \eqref{uniformM},  we can also assume that the sequence $\{u_{\ell,n}\}_{n\in\mathbb{N}}$ is $\ast-$weakly convergent to $u_{\ell}$ in $L^\infty(\Omega)$. In particular, we have $u_\ell\in L^\infty(\Omega)$, as well. 
\par
Hence, we can test the weak formulation of \eqref{autosalone},  
with the function $\zeta\, (u_{\ell, n}-u_{\ell})\in W^{1,p}_0(\Omega;V)$. This is admissible by Lemma \ref{lm:staredentro}. We obtain
\[
\begin{split}
\int_\Omega \langle |\nabla u_{\ell,n}|^{p-2}\,\nabla u_{\ell,n},\nabla (u_{\ell,n}-u_{\ell})\rangle\,\zeta\,dx&=\lambda_{\ell,p}^{\rm LS}(\Omega;V_n)\,\int_\Omega |u_{\ell,n}|^{p-2}\,u_{\ell,n}\,(u_{\ell,n}-u_\ell)\,\zeta\,dx\\
&-\int_\Omega V_n\,|u_{\ell,n}|^{p-2}\,u_{\ell,n}\,(u_{\ell,n}-u_\ell)\,\zeta\,dx\\
&-\int_\Omega \langle |\nabla u_{\ell,n}|^{p-2}\,\nabla u_{\ell,n},\nabla \zeta\rangle\,(u_{\ell,n}-u_\ell)\,dx.
\end{split}
\]
By using the uniform bound \eqref{fratocuggino} and the strong $L^p(\Omega)$ convergence, it is not difficult to see that the first and third terms on the right-hand side converge to $0$, as $n$ goes to $\infty$. As for the second one, we observe that by using the properties of $\zeta$ and the $L^\infty$ bounds \eqref{uniformM}, we have
\[
\begin{split}
\left|\int_\Omega V_n\,|u_{\ell,n}|^{p-2}\,u_{\ell,n}\,(u_{\ell,n}-u_\ell)\,\zeta\,dx\right|&\le \varepsilon_n\,\int_\Omega V\,\zeta\, |u_{\ell, n}|^{p-1}\,|u_{\ell,n}-u_\ell|\,dx\\
&\leq \varepsilon_n\,M_\ell^{p-1}\,\left(M_\ell+\|u_\ell\|_{L^\infty(\Omega)}\right)\,\int_{B_{2R}} V\,dx.
\end{split}
\]
By using that $V\in L^1_{\rm loc}(\mathbb{R}^N)$, we get that this integral converges to $0$, as well.
Thus, we obtain
\[
\lim_{n\to\infty}\int_\Omega \langle |\nabla u_{\ell,n}|^{p-2}\,\nabla u_{\ell,n},\nabla (u_{\ell,n}-u_{\ell})\rangle\,\zeta\,dx=0.
\]
On the other hand, we also get
\[
\lim_{n\to\infty}\int_\Omega \langle |\nabla u_{\ell}|^{p-2}\,\nabla u_{\ell},\nabla (u_{\ell,n}-u_{\ell})\rangle\,\zeta\,dx=0,
\]
which simply follows from the weak convergence in $L^p(\Omega)$ of the gradients. We can thus infer
\[
\lim_{n\to\infty}\int_\Omega \langle |\nabla u_{\ell,n}|^{p-2}\,\nabla u_{\ell,n}-|\nabla u_{\ell}|^{p-2}\,\nabla u_{\ell},\nabla (u_{\ell,n}-u_{\ell})\rangle\,\zeta\,dx=0,
\]
as well. This in particular implies that 
\begin{equation}
\label{1}
\lim_{n\to \infty} \|\nabla u_{\ell,n}-\nabla u_\ell\|_{L^p(B_R)}=0,\qquad \text{for every}\ R>0,
\end{equation}
by standard monotonicity inequalities for the convex power $z\mapsto |z|^p$.
\par
We now wish to pass to the limit in the identity
\[
\int_\Omega \langle |\nabla u_{\ell,n}|^{p-2}\,\nabla u_{\ell,n},\nabla\varphi\rangle\,dx+\int_\Omega V_n\,|u_{\ell,n}|^{p-2}\,u_{\ell,n}\,\varphi\,dx=\lambda_{k,p}^{\rm LS}(\Omega;V_n)\,\int_\Omega |u_{\ell,n}|^{p-2}\,u_{\ell,n}\,\varphi\,dx,
\]
for every $\varphi\in C^\infty_0(\Omega)$. For the leftmost term, it is sufficient to use \eqref{1}, since $\varphi$ is compactly supported. For the rightmost one, we can use the strong $L^p(\Omega)$ convergence and \eqref{vaiaffa}.
Finally, we notice that 
\[
\lim_{n\to\infty} \int_\Omega V_n\,|u_{\ell,n}|^{p-2}\,u_{\ell,n}\,\varphi\,dx=0,\qquad \mbox{for every}\ \varphi\in C^\infty_0(\Omega).
\]
Indeed, again in view of the uniform $L^\infty$ bounds \eqref{uniformM}, we have
\[
\begin{split}
\left|\int_\Omega V_n\,|u_{\ell,n}|^{p-2}\,u_{\ell,n}\,\varphi\,dx\right|&\le \varepsilon_n\,\int_\Omega V\,|u_{\ell,n}|^{p-1}\,|\varphi|\,dx\le \varepsilon_n\,\|\varphi\|_{L^\infty(\Omega)}\,M_\ell^{p-1}\int_{\mathrm{spt}(\varphi)} V\,dx,
\end{split}
\]
and the last quantity converges to $0$, as $n$ goes to $\infty$. 
\par
As a result, by taking the limit as $n$ goes to $\infty$, we get
\[
\int_\Omega \langle |\nabla u_{\ell}|^{p-2}\,\nabla u_{\ell},\nabla\varphi\rangle\,dx=\lambda_{\ell,p}^{\rm LS}(\Omega)\,\int_\Omega |u_{\ell}|^{p-2}\,u_{\ell}\,\varphi\,dx,\qquad \text{for every}\ \varphi\in C^\infty_0(\Omega).
\]
Since $u_\ell\not\equiv0$, we have proved that $u_\ell$ is an eigenfunction of the Dirichlet $p-$Laplacian on $\Omega$, with associated eigenvalue $\lambda_{\ell,p}^{\rm LS}(\Omega)$.
\vskip.2cm\noindent
{\it Step 2: decay at infinity}.  For $\ell\in\{1,\dots,k\}$,  if  $u_{\ell}$ is an  eigenfunction associated  to  $\lambda_{\ell,p}^{\rm LS}(\Omega)$, then  $|u_{\ell}|$  is a subsolution of \eqref{subsol} with $\lambda=\lambda_{\ell,p}^{\rm LS}(\Omega)<\mathcal E_p(\Omega)$. Thanks to Proposition \ref{prop:lambdadecay}, it satisfies  the desired exponential decay estimate. Observe that since eigenfunctions of the $p-$Laplacian are $C^{1,\alpha}$ (see for example the classical reference \cite{DiB} and more recently \cite[Theorem 1.1]{An}), we can assure that the decay estimate actually holds everywhere, and not only almost everywhere.

\section{Some examples}
\label{sec:6}

\subsection{Sets with a massive core}

Let $\Omega\subseteq\mathbb{R}^{N}$ be an unbounded connected open set, with the following property: there exists an open bounded subset $\Omega_0\subsetneq\Omega$ (the ``massive core'') such that
\begin{equation}
\label{massivecore}
\lambda^{\rm LS}_{k, p}(\Omega_0)< \mathcal{E}_p(\Omega),\qquad \text{for some}\ k\in\mathbb{N}\setminus \{0\}.
\end{equation}
By domain monotonicity, the property \eqref{massivecore} in turn implies
\[
\lambda^{\rm LS}_{k,p}(\Omega)<\mathcal{E}_p(\Omega).
\]
Hence our main Main Theorem applies, ensuring  that, for every $\ell\in \{1,\dots, k\}$,  the minmax value $\lambda^{LS}_{\ell, p}(\Omega)$ is an eigenvalue of the Dirichlet $p-$Laplacian on $\Omega$.
\vskip.2cm\noindent
As a simple explicit example of this situation, we may consider an infinite slab with a sufficiently large ball attached to it: such a ball represents the ``massive core''. Indeed, if we take
\[
\Omega_0=B_R\qquad  \text{and}\qquad \Omega = \left(\left(-\frac{1}{2},\frac{1}{2}\right)\times \mathbb{R}^{N-1}\right)\cup \Omega_0,
\]
then, it is rather easy to show that
\[
\lambda^{\rm LS}_{k, p}(\Omega_0) = R^{-p}\,\lambda^{\rm LS}_{k, p}(B_1),\qquad\mathcal{E}_{p}(\Omega) = \lambda_{1,p}\left(\left(-\frac{1}{2},\frac{1}{2}\right)\right)=:\big(\pi_p\big)^p.
\]
Thus, for every given $k\in\mathbb{N}\setminus\{0\}$, condition \eqref{massivecore} is satisfied whenever the radius $R$ is chosen so that
\[
\frac{\Big(\lambda^{\rm LS}_{k,p}(B_1)\Big)^\frac{1}{p}}{\pi_p}< R.
\]
Observe that for this set $\Omega$, the embedding $W^{1,p}_0(\Omega)\hookrightarrow L^p(\Omega)$ is not compact.

\subsection{Curved waveguides}

There is a wide literature devoted to the analysis of the (linear) Spectral Theory of curved waveguides, starting from the pioneering work of Exner and \v{S}eba \cite{ES} (see also the paper \cite{GJ} by Goldstone and Jaffe). 
\par
The landmark result of \cite[Theorem 4.1]{ES} can be summarized as follows: while a straight planar strip $\Omega$ has purely essential spectrum (and thus the bottom of the spectrum $\lambda_1(\Omega)$ is not attained), it is sufficient to bend it ``a little bit'' in order to produce a new domain $\widetilde \Omega$ for which $\lambda_1(\widetilde \Omega)$  is attained by a positive eigenfunction. Thus, the geometric modification produces (at least) an eigenvalue beneath the bottom of the essential spectrum. We refer to \cite{AE} for a related result, aiming at giving a quantitative lower bound on this curvature-induced eigenvalue, as well as to the book \cite{EK} for a systematic study of this phenomenon, considering also the higher dimensional case.
\vskip.2cm\noindent
Phenomena of this type in the nonlinear setting are much less studied:
to the best of our knowledge,  the only results of this type are contained in~\cite{BK} by Baldelli and Krej\v{c}i\v{r}\'ik. More precisely, the authors consider open sets of the form 
\begin{equation}
\label{omegabaldelli}
\Omega= \left\{\Gamma(s)+\sum_{j=1}^{N-1} t_j\, \nu_j(s)\, :\, s\in\mathbb{R},\ \mathbf{t}:=(t_1,\dots, t_{N-1})\in  B' \right\},
\end{equation}
where:
\begin{itemize}
\item $\Gamma: \mathbb{R}\to \mathbb{R}^{N}$ is a $C^{1,1}$ curve such that $|\Gamma'(s)|=1$ for every $s\in\mathbb{R}$, whose curvature function $\kappa$ belongs to $L^{\infty}(\mathbb{R})$;
\vskip.2cm
\item each $\nu_j\in \mathbb{R}\to\mathbb{R}^N$ is almost everywhere differentiable, with 
\[
\langle \nu_j(s),\nu_i(s)\rangle=\delta_{ij},\qquad \langle \nu_j(s),\Gamma'(s)\rangle=0, \qquad \text{for every}\ i,j\in\{1,\dots,N-1\},\ s\in\mathbb{R};
\] 
\item  $B'$ is either a $(N-1)-$dimensional ball or an $(N-1)-$dimensional spherical shell, in both cases centered at the origin.
\end{itemize}
Moreover, the following {\it non-overlapping condition} is assumed
\begin{equation}
\label{nonne}
 \sup_{\mathbf{t}\in B'} |\mathbf{t}|\,\|\kappa\|_{L^{\infty}(\mathbb{R})}<1.
\end{equation}
In a nutshell, the set $\Omega$ is a (sufficiently thin) smooth tubular neighborhood of the image of the curve $\Gamma$.
\par
Then, in \cite[Corollary 1]{BK} it is shown that if $\Omega$ is a {\it bent tube}, i.e. $\kappa\neq 0$ , then the following result holds:
\begin{equation}
\label{baldellicond}
\lim_{|s|\to \infty}\kappa(s)=0\qquad \Longrightarrow \qquad \lambda_{1,p}(\Omega)<\mathcal{E}_p(\Omega).
\end{equation}
We can thus apply our Main Theorem and assures that the value $\lambda_{1,p}(\Omega)$ is attained by some positive eigenfunction, i.e. the nonlinear counterpart of the Exner-\v{S}eba result holds true.
\subsection{The infinite whip}
We now exhibit an example of an unbounded {\it connected} open set $\Omega\subseteq\mathbb{R}^{2}$ with the following properties:
\begin{itemize}
\item  $\mathcal{E}_p(\Omega)<+\infty$
\vskip.2cm
\item it holds 
\[
\lambda^{\rm LS}_{k, p}(\Omega)< \mathcal{E}_p(\Omega),\qquad \text{for every}\ k\in\mathbb{N}\setminus \{0\}.
\]
\end{itemize}  
Thus, according to our Main Theorem, each of the infinitely many minmax values $\lambda^{\rm LS}_{k, p}(\Omega)$ is actually an eigenvalue of the Dirichlet $p-$Laplacian. 
\vskip.2cm\noindent
Such a set will be constructed in the form of a planar curved waveguide (see the previous subsection).
More precisely, we define the $C^{1,1}$ curve
\[
\gamma_{n}(t)=\left\{\begin{array}{rl}\left(t,\dfrac{1+\cos t}{n+1}\right), & \text{if}\ t\in (-\pi , \pi),\\
\\
(t,0),& \text{if}\ |t|\geq \pi,\end{array},
\right.
\]
and first consider a family of open sets $\Omega_{n}$  defined according to  \eqref{omegabaldelli}, by taking     
\[
B'=\left(-\frac{1}{2}, \frac{1}{2}\right)\qquad\text{and}\ \qquad \Gamma_{n}(s)=\gamma_n(\phi_n(s)).
\]
Here $\phi_n:\mathbb{R}\to\mathbb{R}$ is the increasing change of variable such that $\phi_n(0)=0$ and $|\Gamma_n'(s)|=1$, for every $s\in\mathbb{R}$.
Observe that $\Gamma_n$ is a $C^{1,1}$ curve.
In this case, the curvature function $\kappa_n$ is given by
\[
\kappa_n(s)=-\frac {(n+1)^2\,\cos \phi_n(s)}{\big((n+1)^2+\sin^2 \phi_n(s)\big)^\frac{3}{2}}, \quad \text{if}\ s\in (-s_0 , s_0),\qquad \text{where}\ s_0:=\int_0^\pi \sqrt{1+\frac{\sin^2 t}{(n+1)^2}}\,dt,
\]
and it vanishes otherwise.
In particular, 
the non-overlapping condition \eqref{nonne} is satisfied for every $n\in \mathbb N$. We also observe that for every $R\ge 2\,\pi$ we have
\[
 (-\infty,-R) \times \left(-\frac{1}{2}, \frac{1}{2}\right)\subseteq \Omega_n\setminus \overline{B_R}\subseteq \mathbb{R} \times \left(-\frac{1}{2},\frac{1}{2}\right). 
\]

By recalling that for an half-strip of the form $(-\infty,b)\times (-a,a)$ we have
\[
\lambda_{1,p}\big((-\infty,b)\times (-a,a)\big)=\lambda_{1,p}\big((-a,a)\big)=\left(\frac{\pi_p}{2\,a}\right)^p,
\]
we then obtain that 
\[
\lambda_{1,p}(\Omega_n\setminus \overline{B_R})=\big(\pi_p\big)^p,\quad \text{for every}\ R\ge 2\,\pi\qquad \text{and thus}\qquad \mathcal{E}_p(\Omega_n)=\big(\pi_p\big)^p,
\]
thanks to the monotonicity of $\lambda_{1,p}$ with respect to the set inclusion.
Thus, by applying \eqref{baldellicond},  we have  that 
\[
\lambda_{1,p}(\Omega_n)<\mathcal{E}_p(\Omega_n)=\big(\pi_p\big)^p, \qquad  \text{for  every}\ n\in \mathbb N.
\] 

A  continuity argument shows that, for every $n\in \mathbb N,$ there exists $L_n>\pi$ such that  the  bounded open set
\[
\mathcal{O}_n:=\Omega_n\cap \Big((-L_n,L_n)\times \mathbb{R}\Big),
\]
still satisfies
\begin{equation}\label{pezzidifrusta}
\lambda_{1,p}(\mathcal{O}_n)< \pi_p^p \qquad \hbox{for every }n\in \mathbb N.
\end{equation}
The desired open set $\Omega$ is now obtained by suitably gluing together the sets $\mathcal{O}_n$, so to build an ``infinite whip", as in Figure \ref{fig:whip}. \begin{figure}
\includegraphics[scale=.3]{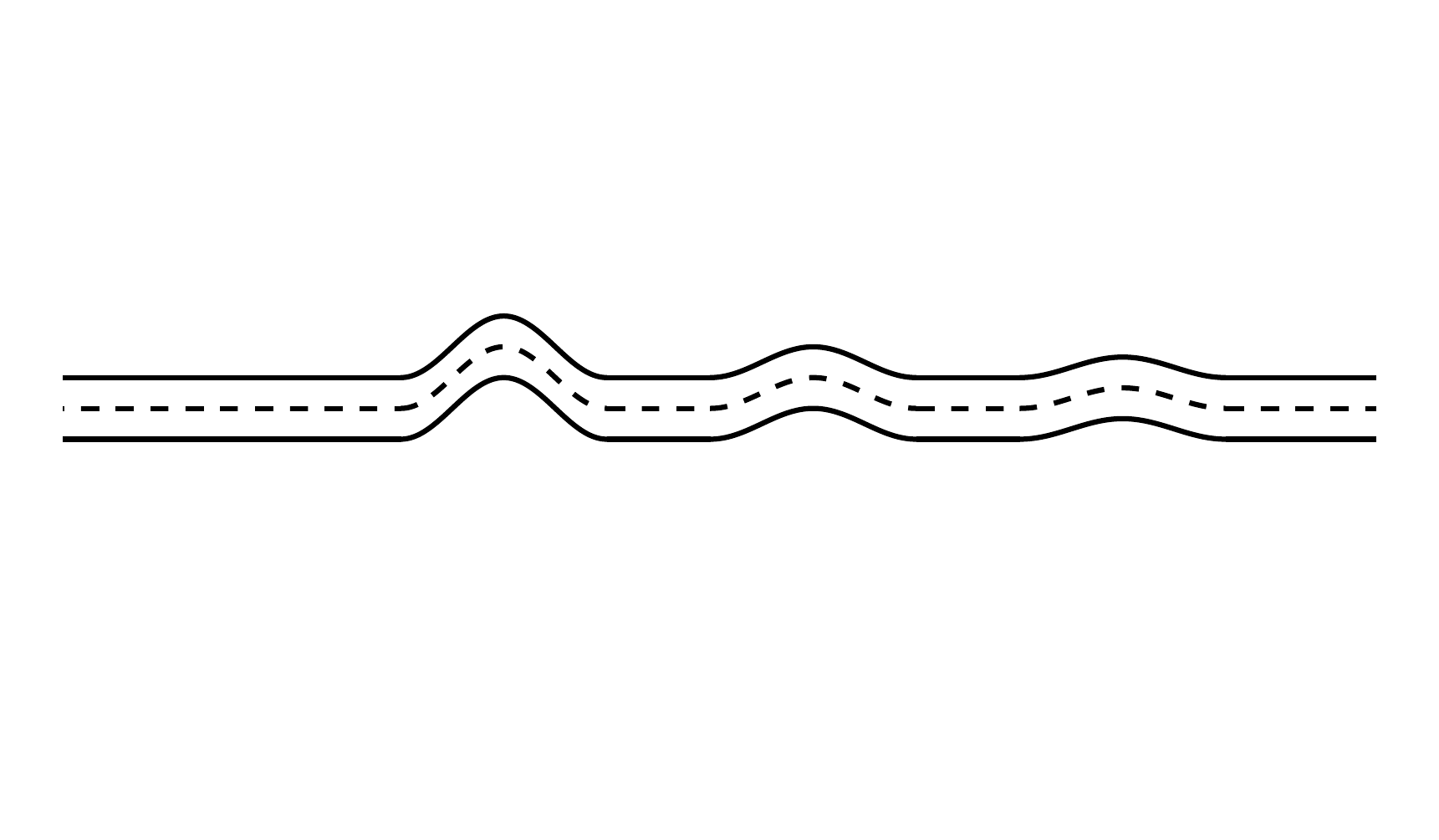}
\caption{The ``infinite whip'' of Lemma \ref{lm:whip}. This can be viewed as the tubular neighborhood of the curve in dashed line, which is ``asymptotically straight'' at infinity.}
\label{fig:whip}
\end{figure}
More precisely, we denote  
\[
\widetilde{\mathcal{O}_n}:=\Omega_n\cap \Big( [-L_n,L_n] \times \mathbb{R}\Big),
\]
and define the translation parameters
\[
\tau_0=L_0\qquad \text{and}\qquad  \tau_n=2\,(L_0+\dots+L_{n-1})+L_n,\quad \text{for}\ n\ge 1.
\]
We then have the following
\begin{lm}[The infinite whip]
\label{lm:whip}
With the previous notation, let us define
\[
\Omega=\left((-\infty, 0)\times \left(-\frac{1}{2}, \frac{1}{2}\right)\right)\cup \left( \bigcup_{n=0}^\infty \left(\widetilde{\mathcal{O}_n}+\tau_n\, \mathbf{e}_1\right) \right),
\]
where $\mathbf{e}_1=(1,0)$. Then we have:
\begin{enumerate}
\item $\mathcal{E}_p(\Omega)=\big(\pi_p\big)^p$;
\vskip.2cm
\item $\lambda^{\rm LS}_{k, p}(\Omega)< \mathcal{E}_p(\Omega)$, for every $k\in\mathbb{N}\setminus \{0\}$.
\end{enumerate}
\end{lm}
\begin{proof}
Let us first compute $\mathcal{E}_p(\Omega)$. By construction, for every $n\in\mathbb{N}$, we have the following inclusions
\[
(-\infty, -\tau_n)\times \left(-\frac{1}{2},\frac{1}{2}\right)\subseteq \Omega\setminus \overline{B_{\tau_n}} \subseteq\mathbb{R}\times\left(-\frac{1}{2},\frac{1}{2}+\frac{2}{n+1}\right),
\] 
which give at once
\[
\left(1+\frac{2}{n+1}\right)^{-p}\,\big(\pi_p\big)^p\le \lambda_{1,p}(\Omega\setminus \overline{B_{\tau_n}}) \le \big(\pi_p\big)^p.
\]
The latter, taking the limit as $n$ goes to $\infty$, proves that $\mathcal{E}_p(\Omega)=(\pi_{p})^p$.
\vskip.2cm\noindent
We can now show point (2) by exploiting \eqref{pezzidifrusta}. Specifically, we fix $k\in\mathbb{N}\setminus\{0\}$ and for every $1\le n\le k$ we take $U_n\in \mathcal{S}_p(\Omega)$ the unique positive eigenfunction associated to $\lambda_{1,p}(\mathcal{O}_n+\tau_n\mathbf{e}_1)$. We extend it to be zero in $\Omega\setminus(\mathcal{O}_n+\tau_n\mathbf{e}_1))$. Observe that such an eigenfunction exists because the set is bounded.
\par
Thanks to the fact that $\{U_1,U_2,\dots,U_k\}$ have  disjoint supports,  we have that they generate a $k-$dimensional vector subspace of $W^{1,p}_0(\Omega)$
\[
K_k:=\mathrm{Vect}\Big(\left\{U_1,\dots,U_k\right\}\Big).
\]
Then $K_k\cap\mathcal{S}_p(\Omega)$ is a compact symmetric subset of $W^{1,p}_0(\Omega)$, with genus equal to $k$ (see for example \cite[Chapter II, Proposition 5.2]{Str}). Moreover, we observe that
\[
\begin{split}
\sum_{i=1}^k \alpha_i\, U_i \in K_k\cap\mathcal{S}_p(\Omega)&\qquad \Longleftrightarrow \qquad\sum_{i=1}^k  |\alpha_i|^p\,  \int_{\Omega}  |U_i|^p\,dx=1\\
&\qquad \Longleftrightarrow\qquad \alpha=(\alpha_1,\dots,\alpha_k)\in\mathbb{S}^{k-1}_{\ell^p}:=\left\{x\in\mathbb{R}^N\, :\, \sum_{i=1}^k |x_i|^p=1\right\}. 
\end{split}
\]
Therefore, for every $n\in\mathbb{N}$, we have
 \[
\begin{split}
\lambda_{k,p}^{\rm LS}(\Omega)\le \max_{\varphi\in K_k\cap \mathcal{S}_p(\Omega)} \int_\Omega |\nabla \varphi|^p\,dx
&=\max_{\alpha\in \mathbb{S}^{k-1}_{\ell^p}} \sum_{i=1}^k |\alpha_i|^p\,\int_\Omega |\nabla U_i|^p\,dx\\
&=\max_{\alpha\in \mathbb{S}^{k-1}_{\ell^p}} \sum_{i=1}^k |\alpha_i|^p\,  \lambda_{1,p}(\mathcal{O}_i) \,\\
&<\max_{\alpha\in \mathbb{S}^{k-1}_{\ell^p}} \sum_{i=1}^k |\alpha_i|^p \, \pi_p^p=\pi_p^p.
\end{split}
\]
Observe that we used the optimality of each function $U_i$ and  \eqref{pezzidifrusta}. We thus have obtained the desired conclusion.
\end{proof}
\appendix

\section{Technical tools}
\label{sec:A}

The following inequality is well-known, we provide a proof for completeness and in order to get a precise control on the constants.
\begin{lm}\label{lm:triangolare}  For every $1<p<\infty$, there exists a constant $c_p>0$ such that 
\begin{equation}
\label{smets0}
\Big||a+b|^p-|a|^p\Big|\le \varepsilon\,|a|^p+\frac{1}{\varepsilon^{p-1}}\,c_p\,|b|^p,\qquad \text{for every}\ a,b\in\mathbb{R}^N\ \text{and every}\ 0<\varepsilon<1.
\end{equation}
\end{lm}
\begin{proof} 
We first write
\[
|a+b|^p-|a|^p=\int_0^1 \frac{d}{dt}|a+t\,b|^p\,dt=p\,\int_0^1 \langle|a+t\,b|^{p-2}\,(a+t\,b),b\rangle\,dt.
\]
By using Cauchy-Schwarz inequality, we thus obtain
\[
\Big||a+b|^p-|a|^p\Big|\le p\,\int_0^1 |a+t\,b|^{p-1}\,|b|\,dt\le p\,\int_0^1 (|a|+t\,|b|)^{p-1}\,|b|\,dt. 
\]
By applying Young's inequality on the last term,  we deduce that
\[
\begin{split}
\Big||a+b|^p-|a|^p\Big|&\le (p-1)\,\delta\,\int_0^1 (|a|+t\,|b|)^p\,dt+\delta^{1-p}\,|b|^p\\
&\le 2^{p-1}\delta\,(p-1)\,(|a|^p+|b|^p)+\delta^{1-p}\,|b|^p.
\end{split}
\]
for every $\delta>0$. In the second inequality, we also used the convexity of the function $\tau\mapsto\tau^p$. Finally, by taking $\varepsilon>0$ and choosing $\delta$ so that
\[
2^{p-1}\,(p-1)\,\delta=\varepsilon\,.
\]
we get the conclusion.
\end{proof}

\begin{lm}
\label{lm:generoditoma}
Let $(X,\|,\cdot\,\|_X)$ and $(Y,\|,\cdot\,\|_Y)$ be two Banach spaces. Let us suppose that $X\subseteq Y$ with continuous inclusion, i.e. there exists a constant $C>0$ such that
\[
\|x\|_Y\le C\,\|x\|_X,\qquad \text{for every}\ x\in X.
\]
Then for every closed set $K\subseteq X$ which is symmetric and compact in the norm topology of $X$, we have
\[
\gamma(K;X)=\gamma(K;Y).
\]
\end{lm}
\begin{proof}
We first observe that each sequence $\{x_n\}_{n\in\mathbb{N}}\subseteq K$ converges in the norm topology of $X$ if and only if it converges in the norm topology of $Y$. One of the two implications is clear, due to the continuity of the inclusion. On the other hand, let us suppose that
\[
\lim_{n\to\infty} \|x_n-\overline{x}\|_Y=0.
\]
By compactness of $K$, we have that $\{x_n\}_{n\in\mathbb{N}}$ admits a subsequence $\{x_{n_k}\}_{k\in\mathbb{N}}$ and point $\overline{y}\in K$ such that
\[
\lim_{k\to\infty} \|x_{n_k}-\overline{y}\|_X=0.
\]
Observe that 
\[
\|\overline{x}-\overline{y}\|_Y\le \lim_{k\to\infty} \|\overline{x}-x_{n_k}\|_Y+C\,\lim_{k\to\infty} \|\overline{y}-x_{n_k}\|_X=0,
\]
so that $\overline{y}=\overline{x}$. Moreover, we can repeat the previous argument for every subsequence $\{x_{n_k}\}_{k\in\mathbb{N}}\subseteq\{x_n\}_{n\in\mathbb{N}}$ and obtain that it always admits a further subsequence, still converging in $X$ to the same limit $\overline{x}$. This finally shows that we have convergence in $X$ of the full sequence, i.e.
\[
\lim_{n\to\infty} \|x_n-\overline{x}\|_X=0,
\]
as well.
\par
Thanks to previous property, for every $k\in\mathbb{N}$ the following fact holds: a function $f:K\to\mathbb{S}^{k-1}$ is continuous with respect to the norm topology of $X$ if and only if it is continuous with respect to that of $Y$. In light of the definition of Krasnosel'ski\u{\i} genus, this is enough to conclude.
\end{proof}

\section{The case $p=2$}
\label{sec:B}

In the Hilbertian case $p=2$, we use the following distinguished notation
\[
\lambda_1(\Omega)=\inf_{\varphi\in \mathcal{S}_2(\Omega)\cap W^{1,2}_0(\Omega)} \int_\Omega |\nabla \varphi|^2\,dx,
\]
as in the Introduction.
In this section, we show that the minmax values constructed in the Main Theorem coincide with the usual Courant-Fischer minmax values
\[
\lambda_{k}(\Omega)=\inf\left\{ \max_{u\in K\cap \mathcal{S}_2(\Omega)} \int_\Omega|\nabla u|^2\,dx\,:\, K \subseteq W^{1,2}_0(\Omega)\ \mbox{subspace with}\ \dim K=k \right\},
\]
recalled in the Introduction. Thus, for $p=2$ our result boils down to the classical case. 

\begin{prop}
\label{prop:tabaccaio}
Let $\Omega\subseteq\mathbb{R}^N$ be an open set and let us suppose that there exists $k\in\mathbb{N}\setminus\{0\}$ such that
\begin{equation}\
\label{PLC}
\lambda_{k,2}^{\rm LS}(\Omega)<\mathcal{E}(\Omega):=\sup_{R>0} \lambda_1(\Omega\setminus \overline{B_R}).
\end{equation}
Then, for every $\ell\in\{1,\dots,k\}$ we have
\[
\lambda_{\ell,2}^{\rm LS}(\Omega)=\lambda_{\ell}(\Omega).
\]
\end{prop}
\begin{proof}
We can limit ourselves to the case $\mathcal{E}(\Omega)<+\infty$, otherwise the result is already contained in \cite[Theorem A.2]{BraParSqu}.
We may suppose that $k\ge 2$, otherwise the claim is straightforward. We start by observing that for every vector subspace $K\subseteq W^{1,2}_0(\Omega)$ having dimension $\ell$, it holds
\[
K\cap\mathcal{S}_2(\Omega)\in \mathcal{W}_{\ell,2}(\Omega),
\] 
i.e. $K\cap\mathcal{S}_2(\Omega)$ is a compact symmetric subset of $W^{1,2}_0(\Omega)$, with genus equal to $\ell$ (see for example \cite[Chapter II, Proposition 5.2]{Str}). This entails that 
\[
\lambda_{\ell,2}^{\rm LS}(\Omega)\le \lambda_\ell(\Omega).
\]
In order to prove the reverse inequality, we will use an approximation argument\footnote{This route will avoid considering the Spectral Resolution of the Dirichlet-Laplacian on $\Omega$ (see for example \cite[Chapter 6]{BS}). Indeed, this last operator could have a continuous part in its  spectrum, under the standing assumptions on the open set. On the contrary, the approximation argument will only need the spectral properties of operators having a discrete spectrum. The reader with a sufficiently good expertise in Spectral Theory will certainly find our argument a bit akward.}. Namely, let us consider the same operators of Proposition \ref{prop:approssimazione}, i.e.
\[
\varphi\mapsto-\Delta\varphi+V_n\,\varphi,
\]
with homogeneous Dirichlet boundary conditions. The quadratic form naturally associated to this operator is given by the functional $\mathcal{G}_n$, defined in Proposition \ref{prop:approssimazione}. By classical Spectral Theory, for every fixed $n\in\mathbb{N}\setminus\{0\}$, the compactness of the embedding $W^{1,2}_0(\Omega;V)\hookrightarrow L^2(\Omega)$ guarantees that such an operator has a discrete spectrum, made of a diverging sequence of positive eigenvalues, each one with finite multiplicity. They can be characterized through the Courant-Fischer formula
\[
\lambda_{\ell}(\Omega;V_n):=\inf\left\{ \max_{u\in K\cap \mathcal{S}_2(\Omega)} \mathcal{G}_n(u)\,:\, K \subseteq W^{1,2}_0(\Omega;V)\ \mbox{subspace with}\ \dim K=\ell \right\}.
\]
Moreover, the infimum on the right-hand side is attained by
\[
K_\ell=\mathrm{Vect}\Big(\left\{u_{1,n},\dots,u_{\ell,n}\right\}\Big),
\]
i.e. by the vector subspace generated by the first $\ell$ eigenfunctions of the operator. In fact, these eigenfunctions can be chosen so that 
\[
\int_\Omega u_{i,n}\,u_{j,n}\,dx=\delta_{ij},\qquad \text{for every}\ n\in\mathbb{N},\ i,j\in\{1,\dots,\ell\},
\]
and the following characterization also holds:
\[
\lambda_{\ell}(\Omega;V_n)=\min_{\varphi\in \mathcal{S}_2(\Omega)\cap W^{1,2}_0(\Omega;V)}\left\{\mathcal{G}_n(\varphi)\,:\, \int_\Omega \varphi\,u_{i,n}\,dx=0,\ i\in\{1,\dots,\ell-1\}\right\}.
\]
We are now ready to conclude the proof: according to \cite[Theorem A.2]{BraParSqu}, for every $\ell\in \mathbb{N}\setminus\{0\}$ we have
\[
\lambda_{\ell,2}^{\rm LS}(\Omega;V_n)=\lambda_{\ell}(\Omega;V_n)=\int_\Omega |\nabla u_{\ell,n}|^2\,dx+\int_\Omega V_n\,|u_{\ell,n}|^2\,dx,
\]
i.e. the equality claimed in the statement holds true for the approximating operator.
By the proof of the Main Theorem, we know that each $\{u_{i,n}\}_{n\in\mathbb{N}}$ converges strongly in $L^2(\Omega)$ to an eigenfunction $u_i$ of the Dirichlet-Laplacian on $\Omega$, up to a subsequence. In particular, we still have 
\begin{equation}
\label{ortaggi}
\int_\Omega u_{i}\,u_{j}\,dx=\delta_{ij},\qquad \text{for every}\ i,j\in\{1,\dots,\ell\}.
\end{equation}
By virtue of Proposition \ref{prop:approssimazione} and by the lower semicontinuity of the Dirichlet integral, we thus get for every $\ell\in\{1,\dots,k\}$
\[
\begin{split}
\lambda_{\ell,2}^{\rm LS}(\Omega)=\lim_{n\to\infty} \lambda_{\ell,2}^{\rm LS}(\Omega;V_n)\ge\lim_{n\to\infty} \int_\Omega |\nabla u_{\ell,n}|^2\,dx\ge \int_\Omega |\nabla u_\ell|^2\,dx\\
\end{split}
\]
On the other hand, by defining the $\ell-$dimensional vector subspace
\[
\overline{K}=\mathrm{Vect}\Big(\left\{u_{1},\dots,u_{\ell}\right\}\Big),
\]
we get
\[
\begin{split}
\lambda_\ell(\Omega)\le \max_{\varphi\in \overline{K}\cap\mathcal{S}_2(\Omega)} \int_\Omega |\nabla\varphi|^2\,dx&=\max_{\alpha=(\alpha_1,\dots,\alpha_\ell)\in \mathbb{S}^{\ell-1}} \int_\Omega \left|\sum_{i=1}^\ell \alpha_i\,\nabla u_i\right|^2\,dx\\
&=\max_{\alpha=(\alpha_1,\dots,\alpha_\ell)\in \mathbb{S}^{\ell-1}} \sum_{i=1}^\ell |\alpha_i|^2\,\int_\Omega |\nabla u_i|^2\,dx\le \int_\Omega |\nabla u_\ell|^2\,dx,
\end{split}
\]
thanks to the orthogonality relations \eqref{ortaggi} and the fact that each $u_i$ is an eigenfunction. In turn, we get the reverse estimate 
\[
\lambda_{\ell,2}^{\rm LS}(\Omega)\ge \lambda_\ell(\Omega),
\]
as desired.
\end{proof}

\end{document}